\documentclass[a4paper,reqno]{amsart}%
\usepackage{pgf,tikz}
\usepackage{amssymb}
\usepackage{amsmath}
\usepackage{amsfonts}
\usepackage{graphicx}
\usepackage{bm}%
\providecommand{\U}[1]{\protect\rule{.1in}{.1in}}
\usetikzlibrary{backgrounds}
\usetikzlibrary{arrows}
\let\oldmathbf\mathbf
\renewcommand{\mathbf}[1]{\boldsymbol{\oldmathbf{#1}}}
\newtheorem{theorem}{Theorem}

\newtheorem{corollary}[theorem]{Corollary}

\newtheorem{definition}[theorem]{Definition}

\newtheorem{lemma}[theorem]{Lemma}

\newtheorem{proposition}[theorem]{Proposition}
\newtheorem{remark}[theorem]{Remark}

\allowdisplaybreaks
\begin{document}
\title{Discrepancy for convex bodies with isolated flat points}
\author[L. Brandolini]{Luca Brandolini}
\address{Dipartimento di Ingegneria Gestionale, dell'Informazione e della Produzione,
Universit\`a degli Studi di Bergamo, Viale Marconi 5, Dalmine BG, Italy}
\email{luca.brandolini@unibg.it}
\author[L. Colzani]{Leonardo Colzani}
\address{Dipartimento di Matematica e Applicazioni, Universit\`a di Milano-Bicocca, Via
Cozzi 55, Milano, Italy}
\email{leonardo.colzani@unimib.it}
\author[B. Gariboldi]{Bianca Gariboldi}
\address{Dipartimento di Ingegneria Gestionale, dell'Informazione e della Produzione,
Universit\`a degli Studi di Bergamo, Viale Marconi 5, Dalmine BG, Italy}
\email{biancamaria.gariboldi@unibg.it}
\author[G. Gigante]{Giacomo Gigante}
\address{Dipartimento di Ingegneria Gestionale, dell'Informazione e della Produzione,
Universit\`a degli Studi di Bergamo, Viale Marconi 5, Dalmine BG, Italy}
\email{giacomo.gigante@unibg.it}
\author[G. Travaglini]{Giancarlo Travaglini}
\address{Dipartimento di Matematica e Applicazioni, Universit\`a di Milano-Bicocca, Via
Cozzi 55, Milano, Italy}
\email{giancarlo.travaglini@unimib.it}
\date{}

\begin{abstract}
We consider the discrepancy of the integer lattice with respect to the
collection of all translated copies of a dilated convex body having a finite
number of flat, possibly non-smooth, points in its boundary. We estimate the
$L^{p}$ norm of the discrepancy with respect to the translation variable as
the dilation parameter goes to infinity. If there is a single flat point with
normal in a rational direction we obtain an asymptotic expansion for this
norm. Anomalies may appear when two flat points have opposite normals. When
all the flat points have normals in generic irrational directions, we obtain a
smaller discrepancy. Our proofs depend on careful estimates for the Fourier
transform of the characteristic function of the convex body.

\end{abstract}
\maketitle

\section{Introduction}

Let $B$ be a convex body in $\mathbb{R}^{d}$, that is a convex bounded set
with nonempty interior, and for every $R>1$ and $\mathbf{z}\in\mathbb{R}^{d}$
let%
\[
D_{R}\left(  \mathbf{z}\right)  =-R^{d}\left\vert B\right\vert +\sum
_{\mathbf{m}\in\mathbb{Z}^{d}}\chi_{RB}\left(  \mathbf{z}+\mathbf{m}\right)
\]
be the discrepancy between the number of integer points inside a dilated and
translated copy of $B$ and its volume. The function $\mathbf{z\mapsto}%
D_{R}\left(  \mathbf{z}\right)  $ is periodic and a straightforward
computation shows that it has the Fourier expansion%
\begin{equation}
\sum_{\mathbf{m}\in\mathbb{Z}^{d}\setminus\left\{  \mathbf{0}\right\}  }%
R^{d}\widehat{\chi}_{B}\left(  R\mathbf{m}\right)  e^{2\pi i\mathbf{m}%
\cdot\mathbf{z}} \label{Fourier discrep}%
\end{equation}
where $\widehat{\chi}_{B}\left(  \mathbf{\zeta}\right)  $ denotes the Fourier
transform of $\chi_{B}\left(  \mathbf{z}\right)  $, that is%
\[
\widehat{\chi}_{B}\left(  \mathbf{\zeta}\right)  =\int_{B}e^{-2\pi
i\mathbf{\zeta}\cdot\mathbf{z}}d\mathbf{z}.
\]
The size of $D_{R}\left(  \mathbf{z}\right)  $ as $R\rightarrow+\infty$ is
therefore closely connected to the decay of $\widehat{\chi}_{B}\left(
\mathbf{\zeta}\right)  $ as $\left\vert \mathbf{\zeta}\right\vert
\rightarrow+\infty$. For example, if the boundary of $B$ is smooth and has
everywhere positive Gaussian curvature then $\widehat{\chi}_{B}\left(
\mathbf{\zeta}\right)  $ has the decay%
\begin{equation}
\left\vert \widehat{\chi}_{B}\left(  \mathbf{\zeta}\right)  \right\vert
\leqslant c\left\vert \mathbf{\zeta}\right\vert ^{-\frac{d+1}{2}}
\label{decay pos Gauss}%
\end{equation}
(see \cite[Chapter 8]{Stein}), and it can be shown that this rate of decay is
optimal. Under the assumption (\ref{decay pos Gauss}), in \cite[Corollary
3]{BCGT} the authors proved the following estimates for the $L^{p}$ norm of
the discrepancy function%
\begin{equation}
\left(  \int_{\mathbb{T}^{d}}\left\vert D_{R}\left(  \mathbf{z}\right)
\right\vert ^{p}d\mathbf{z}\right)  ^{1/p}\leqslant\left\{
\begin{array}
[c]{ll}%
cR^{\frac{d-1}{2}} & 1\leqslant p<2d/\left(  d-1\right)  ,\\
cR^{\frac{d-1}{2}}\log^{\frac{d-1}{2d}}\left(  R\right)  & p=2d/\left(
d-1\right)  ,\\
cR^{\frac{d\left(  d-1\right)  }{\left(  d+1\right)  }\left(  1-\frac{1}%
{p}\right)  } & p>2d/\left(  d-1\right)  .
\end{array}
\right.  \label{Lp-discrep}%
\end{equation}
In \cite[Theorem 5]{BCGT} it has also been shown that the above estimates are
sharp in the range $1\leqslant p<2d/\left(  d-1\right)  $. More precisely,
using the asymptotic expansion for $\widehat{\chi}_{B}\left(  \mathbf{\zeta
}\right)  $, it has been proved that when $B$ is not symmetric about a point
or $d\not \equiv 1\left(  \operatorname{mod}4\right)  $ one has, for every
$p\geqslant1$,%
\[
\left(  \int_{\mathbb{T}^{d}}\left\vert D_{R}\left(  \mathbf{z}\right)
\right\vert ^{p}d\mathbf{z}\right)  ^{1/p}\geqslant cR^{\frac{d-1}{2}}.
\]
On the other hand, when $B$ is symmetric about a point and $d\equiv1\left(
\operatorname{mod}4\right)  $,%
\begin{align*}
\limsup_{R\rightarrow+\infty}R^{-\frac{d-1}{2}}\left(  \int_{\mathbb{T}^{d}%
}\left\vert D_{R}\left(  \mathbf{z}\right)  \right\vert ^{p}d\mathbf{z}%
\right)  ^{1/p}  &  >0~~~~\text{for every }p\geqslant1,\\
\liminf_{R\rightarrow+\infty}R^{-\frac{d-1}{2}}\left(  \int_{\mathbb{T}^{d}%
}\left\vert D_{R}\left(  \mathbf{z}\right)  \right\vert ^{p}d\mathbf{z}%
\right)  ^{1/p}  &  =0~~~~\text{for every }p<\frac{2d}{d-1}.
\end{align*}

Up to now we have considered the case of positive Gaussian curvature.

When the Gaussian curvature of the boundary of $B$ vanishes at some point the
estimate (\ref{decay pos Gauss}) fails and the rate of decay depends on the direction.

More precisely the decay of the Fourier transform (\ref{decay pos Gauss})
holds in a given direction $\Theta$ if the Gaussian curvature does not vanish
at the points on the boundary of $B$ where the normal is $\pm\Theta$. When the
curvature vanishes the rate of decay of $\widehat{\chi}_{B}\left(  \rho
\Theta\right)  $ can be significantly smaller. We will see that in this case
the behavior of the $L^{p}$ norms of the discrepancy function may differ from
the case of positive Gaussian curvature.

To the authors' knowledge the discrepancy for convex bodies with vanishing
Gaussian curvature has been considered only for specific classes of convex
bodies and only for $L^{\infty}$ estimates. See e.g. \cite{CdV}, \cite{Guo1},
\cite{Guo2}, \cite{Guo3}, \cite{ISS}, \cite{K}, \cite{Kr}, \cite{R1},
\cite{R2}. See also \cite{K} for an estimate from below of the $L^{2}$
discrepancy associated to the curve $x^{2}+y^{4}=1$. \bigskip

Throughout the paper we will use bold symbols only for $d$-dimensional points
and non-bold symbol for lower dimensional points. Moreover when we write a
point $\mathbf{z}=(x,t)$ or $\mathbf{\zeta}=\left(  \xi,s\right)  $ we agree
that $x,\xi\in\mathbb{R}^{d-1}$ and $t,s\in\mathbb{R}$.\bigskip

We are happy to thank Gabriele Bianchi for some interesting remarks on the
geometric properties of the convex bodies considered in this paper (see
\cite{Bi}).

\section{Statements of the results}

In this paper we study the $L^{p}$ norms of the discrepancy function
associated to a convex body whose boundary has a finite number of isolated
flat points. The relevant example is a convex body $B$ such that $\partial B$
has everywhere positive Gaussian curvature except at the origin and such that,
in a neighborhood of the origin, $\partial B$ is the graph of the function
$t=\left\vert x\right\vert ^{\gamma}$, with $x\in\mathbb{R}^{d-1}$ and some
$\gamma>2$. This function is smooth at the origin only when $\gamma$ is a
positive even integer, and the geometric control of the Fourier transform in
\cite{BNW} does not apply directly.

We are actually interested in a larger class of convex bodies and this is why
we introduce the following definition.

\begin{definition}
Let $U$ be a bounded open neighborhood of the origin in $\mathbb{R}^{d-1}$,
let $\Phi\in C^{\infty}\left(  U\setminus\left\{  0\right\}  \right)  $ and
let $\gamma>1$. For every $x\in U\setminus\left\{  0\right\}  $ let $\mu
_{1}\left(  x\right)  ,\ldots,\mu_{d-1}\left(  x\right)  $ be the eigenvalues
of the Hessian matrix of $\Phi$. We say that $\Phi\in S_{\gamma}\left(
U\right)  $ if for $j=1,\ldots,d-1,$%
\begin{equation}
0<\inf_{x\in U\setminus\left\{  0\right\}  }\left\vert x\right\vert
^{2-\gamma}\mu_{j}\left(  x\right)  \label{Eigenvalues}%
\end{equation}
and, for every multi-index $\alpha,$%
\begin{equation}
\sup_{x\in U\setminus\left\{  0\right\}  }\left\vert x\right\vert ^{\left\vert
\alpha\right\vert -\gamma}\left\vert \frac{\partial^{\left\vert \alpha
\right\vert }\Phi}{\partial x^{\alpha}}\left(  x\right)  \right\vert <+\infty.
\label{Deriv-Phi}%
\end{equation}

\end{definition}

Observe that if $\Phi\in S_{\gamma}\left(  U\right)  $ then for some
$c_{1},c_{2}>0$,%
\[
c_{1}\left\vert x\right\vert ^{\gamma-2}\leqslant\mu_{j}\left(  x\right)
\leqslant c_{2}\left\vert x\right\vert ^{\gamma-2}.
\]

Moreover, since $\gamma>1,$ we have $\Phi\in C^{1}\left(  U\right)  $.

\begin{definition}
Let $B$ be a convex body in $\mathbb{R}^{d}$ and let $\mathbf{z}\in\partial B$
and let $\gamma>1$. We say that $\mathbf{z}$ is an isolated flat point of
order $\gamma$ if, in a neighborhood of $\mathbf{z}$ and in a suitable
Cartesian coordinate system with the origin in $\mathbf{z}$, $\partial B$ is
the graph of a function $\Phi\in S_{\gamma}\left(  U\right)  $, as in the
previous definition.
\end{definition}

Convex bodies with flat points can be easily constructed by taking powers of
strictly convex functions.

\begin{proposition}
\label{Prop H}Let $U$ be a bounded open neighborhood of the origin in
$\mathbb{R}^{d-1}$, let $H\in C^{\infty}\left(  U\right)  $ such that
$H\left(  0\right)  =0$, $\nabla H\left(  x\right)  =0$ and assume that its
Hessian matrix is positive definite at the origin. Let $\gamma>1$. Then the
function $\Phi\left(  x\right)  =$ $\left[  H\left(  x\right)  \right]
^{\gamma/2}\in S_{\gamma}\left(  U\right)  $.
\end{proposition}

We have already observed that some of the results in this paper for the
singularity $\left\vert x\right\vert ^{\gamma}$ with $\gamma$ even integer are
not new. However observe that $\left\vert x\right\vert ^{2n}$ is analytic,
while the above definition does not imply that the boundary is smooth. For
example, in dimension $2$ consider a singularity of the kind $\Phi
^{\prime\prime}\left(  x\right)  =2+\sin\left(  \log\left(  \left\vert
x\right\vert \right)  \right)  $.

Interestingly, in the following Proposition \ref{Prop Stima chi} concerning te
decay of the Fourier transform of a convex body with a flat point of order
$\gamma$, the case $\gamma=2$ with non smooth flat points requires some extra care.

The discrepancy for convex bodies with flat points in the above class is
described by the following theorem.

\begin{theorem}
\label{Theorem A}Let $B$ be a bounded convex body in $\mathbb{R}^{d}$. Assume
that $\partial B$ is smooth with everywhere positive Gaussian curvature except
for a finite number of isolated flat points of order at most $\gamma$.

1) For $1<\gamma\leqslant2$ we have%
\[
\left(  \int_{\mathbb{T}^{d}}\left\vert D_{R}\left(  \mathbf{z}\right)
\right\vert ^{p}d\mathbf{z}\right)  ^{1/p}\leqslant\left\{
\begin{array}
[c]{ll}%
cR^{\frac{d-1}{2}} & 1\leqslant p<2d/\left(  d-1\right)  ,\\
cR^{\frac{d-1}{2}}\log^{\frac{d-1}{2d}}\left(  R\right)  & p=2d/\left(
d-1\right)  ,\\
cR^{\frac{d\left(  d-1\right)  }{\left(  d+1\right)  }\left(  1-\frac{1}%
{p}\right)  } & p>2d/\left(  d-1\right)  .
\end{array}
\right.
\]

2) For $2<\gamma\leqslant d+1$ we have%
\[
\left(  \int_{\mathbb{T}^{d}}\left\vert D_{R}\left(  \mathbf{z}\right)
\right\vert ^{p}d\mathbf{z}\right)  ^{1/p}\leqslant\left\{
\begin{array}
[c]{cl}%
cR^{\left(  d-1\right)  \left(  1-\frac{1}{\gamma}\right)  } & 1\leqslant
p\leqslant\left(  2d\right)  /\left(  d+1-\gamma\right) \\
cR^{\frac{d\left(  d-1\right)  }{d+1}\left(  1-\frac{2}{\gamma p}\right)  } &
p>\left(  2d\right)  /\left(  d+1-\gamma\right)
\end{array}
\right.
\]

3) For $\gamma>d+1$ and every $p\geqslant1$ we have%
\[
\left(  \int_{\mathbb{T}^{d}}\left\vert D_{R}\left(  \mathbf{z}\right)
\right\vert ^{p}d\mathbf{z}\right)  ^{1/p}\ \leqslant cR^{\left(  d-1\right)
\left(  1-\frac{1}{\gamma}\right)  }.
\]

\end{theorem}

The picture summarizes our estimates for the discrepancy.

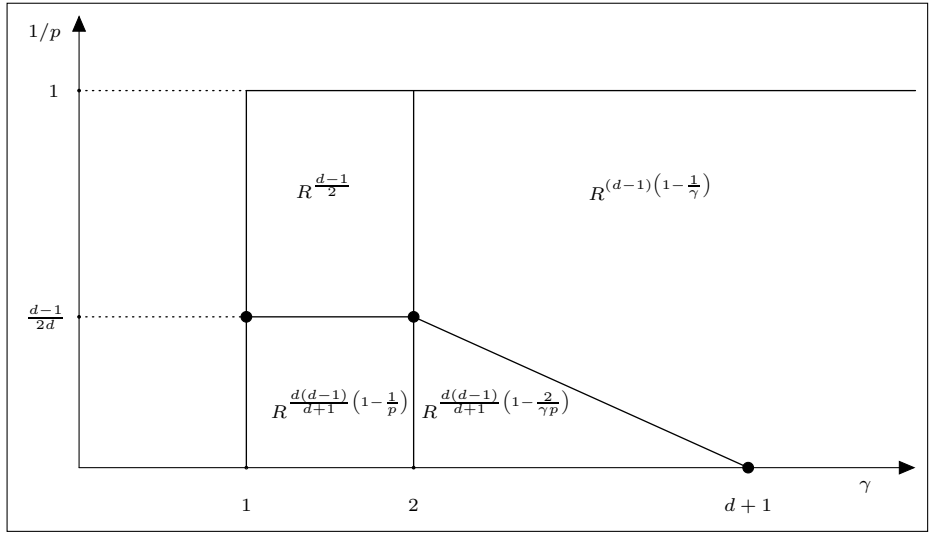
\begin{center}
\begin{tikzpicture}[framed,line cap=round,line join=round,>=triangle 45,x=2.2cm,y=5cm]
\draw[->,color=black] (0,0.) -- (5.,0.);
\draw[->,color=black] (0.,0) -- (0.,1.2);
\draw [line width=0.5pt] (1.,0.4)-- (2.,0.4);
\draw [line width=0.5pt] (2.,0.4)-- (4.,0.);
\draw [line width=0.5pt] (2.,0.4)-- (2.,0.);
\draw [line width=0.5pt] (1.,0.4)-- (1.,0.);1
\draw [line width=0.5pt] (1.,1.)-- (1.,0.4);
\draw [line width=0.5pt] (1.,1.)-- (5.,1.);
\draw [line width=0.5pt] (2.,1.)-- (2.,0.4);
\draw [line width=0.5pt,dotted] (0.,0.4)-- (1.,0.4);
\draw [line width=0.5pt,dotted] (0.,1.)-- (1.,1.);
\begin{scriptsize}
\draw (1.1,0.23) node[anchor=north west] {$R^{\frac{d(d-1)}{d+1}\left(1-\frac1p\right)}$};
\draw (1.25,0.8) node[anchor=north west] {$R^{\frac{d-1}2}$};
\draw (3.,0.8) node[anchor=north west] {$R^{(d-1)\left(1-\frac 1 \gamma\right)}$};
\draw (2,0.23) node[anchor=north west] {$R^{\frac{d(d-1)}{d+1}\left(1-\frac 2{\gamma p}\right)}$};
\draw (4.7,-0.05) node {$\gamma$};
\draw (-0.35,1.2) node[anchor=north west] {$1/p$};
\draw (1.,0.) circle (0.5pt);
\draw[color=black] (1,-0.1) node {$1$};
\draw [fill=black] (2.,0.) circle (0.5pt);
\draw[color=black] (2.,-0.1) node {$2$};
\draw [fill=black] (4.,0.) circle (2.0pt);
\draw[color=black] (4,-0.1) node {$d+1$};
\draw [fill=black] (0.,1.) circle (0.5pt);
\draw[color=black] (-0.15,1) node {$1$};
\draw [fill=black] (0.,0.4) circle (0.5pt);
\draw[color=black] (-0.2,0.4) node {$\frac{d-1}{2d}$};
\draw [fill=black] (1.,0.4) circle (2.0pt);
\draw [fill=black] (2.,0.4) circle (2.0pt);
\end{scriptsize}
\end{tikzpicture}

\end{center}

The proof of the above theorem relies on precise estimate for the Fourier
transform of $\chi_{B}\left(  \mathbf{z}\right)  $. See Proposition
\ref{Prop Stima chi} below. The estimates in point 1) are the same as in
\cite{BCGT} for the case of positive Gaussian curvature and are independent of
$\gamma$. On the contrary, as we will see from the proof, in the cases 2) and
3) the flat points give the main contribution. In the case 2) if $p=+\infty$
then%
\[
\frac{d\left(  d-1\right)  }{d+1}\left(  1-\frac{2}{p\gamma}\right)
=\frac{d\left(  d-1\right)  }{d+1}.
\]
Hence, when $\gamma\leqslant d+1$ the estimates for the $L^{\infty}$
discrepancy of $B$ match Landau's estimates for the $L^{\infty}$ discrepancy
of the ball. See e.g. \cite{PST}. For $p\geq2d/\left(  d-1\right)  $ the above
result extends a theorem of Colin de Verdiere \cite{CdV}.

In the next theorem we consider convex bodies with a flat point with normal
pointing in a rational direction. In this case, some of the previous estimates
can be improved to asymptotic estimates.

\begin{theorem}
\label{Theorem B}Let $B$ be a bounded convex body in $\mathbb{R}^{d}$. Assume
that $\partial B$ is smooth with everywhere positive Gaussian curvature except
at most at two points $P$ and $Q$ with outward unit normals $-\Theta$ and
$\Theta$ which are flat of order $\gamma_{P}$ and $\gamma_{Q}$ respectively.
Let%
\[
S\left(  t\right)  =\left\vert \left\{  \mathbf{z}\in B:\mathbf{z}\cdot
\Theta=t\right\}  \right\vert
\]
be $\left(  d-1\right)  $-dimensional measure of the slices of $B$ that are
orthogonal to $\Theta$. The function $S\left(  t\right)  $ is supported in
$P\cdot\Theta\leqslant t\leqslant Q\cdot\Theta$ and is known to be smooth in
$P\cdot\Theta<t<Q\cdot\Theta$. Assume that there exist two smooth functions
$G_{P}\left(  r\right)  $ and $G_{Q}\left(  r\right)  $ with $G_{P}\left(
0\right)  \neq0$ and $G_{Q}\left(  0\right)  \neq0$ such that, for
$u\geqslant0$ sufficiently small%
\[
S\left(  P\cdot\Theta+u\right)  =u^{\frac{d-1}{\gamma_{P}}}G_{P}\left(
u^{1/\gamma_{P}}\right)
\]
and%
\[
S\left(  Q\cdot\Theta-u\right)  =u^{\frac{d-1}{\gamma_{Q}}}G_{Q}\left(
u^{1/\gamma_{Q}}\right)  .
\]
Finally, assume that the direction $\Theta$ is rational, that is $\alpha
\Theta\in\mathbb{Z}^{d}$ for some $\alpha$, and denote by $\mathbf{m}_{0}$ the
first non-zero integer point in the direction $\Theta$. Define%
\[
A_{P}\left(  \mathbf{z}\right)  =\frac{2G_{P}\left(  0\right)  \Gamma\left(
\frac{d-1}{\gamma_{P}}+1\right)  }{\left(  2\pi\left\vert \mathbf{m}%
_{0}\right\vert \right)  ^{\frac{d-1}{\gamma_{P}}+1}}\sum_{k=1}^{+\infty
}k^{-1-\frac{d-1}{\gamma_{P}}}\sin\left(  2\pi k\mathbf{m}_{0}\cdot
\mathbf{z}-\frac{\pi}{2}\frac{d-1}{\gamma_{P}}\right)  ,
\]
and%
\[
A_{Q}\left(  \mathbf{z}\right)  =-\frac{2G_{Q}\left(  0\right)  \Gamma\left(
\frac{d-1}{\gamma_{Q}}+1\right)  }{\left(  2\pi\left\vert \mathbf{m}%
_{0}\right\vert \right)  ^{\frac{d-1}{\gamma_{Q}}+1}}\sum_{k=1}^{+\infty
}k^{-1-\frac{d-1}{\gamma_{Q}}}\sin\left(  2\pi k\mathbf{m}_{0}\cdot
\mathbf{z}+\frac{\pi}{2}\frac{d-1}{\gamma_{Q}}\right)  .
\]

1) Let $\gamma_{P}>\gamma_{Q}\geqslant2$ and assume that one of the two
alternatives holds:%
\[%
\begin{array}
[c]{c}%
2<\gamma_{P}\leqslant d+1\text{ and }p<\left(  2d\right)  /\left(
d+1-\gamma_{P}\right)  ,\\
\text{or}\\
\gamma_{P}>d+1~\text{and}~p\leqslant+\infty.
\end{array}
\]
Then there exist constants $\delta>0$ and $c>0$ such that for every
$R\geqslant1$,%
\[
\left(  \int_{\mathbb{T}^{d}}\left\vert D_{R}\left(  \mathbf{z}\right)
-R^{\left(  d-1\right)  \left(  1-1/\gamma_{P}\right)  }A_{P}\left(
\mathbf{z-}RP\right)  \right\vert ^{p}d\mathbf{z}\right)  ^{1/p}\ \leqslant
cR^{\left(  d-1\right)  \left(  1-1/\gamma_{P}\right)  -\delta}.
\]
In particular, as $R\rightarrow+\infty$, we have the following asymptotic%
\[
\left(  \int_{\mathbb{T}^{d}}\left\vert D_{R}\left(  \mathbf{z}\right)
\right\vert ^{p}d\mathbf{z}\right)  ^{1/p}\sim R^{\left(  d-1\right)  \left(
1-1/\gamma_{P}\right)  }\left(  \int_{\mathbb{T}^{d}}\left\vert A_{P}\left(
\mathbf{z}\right)  \right\vert ^{p}d\mathbf{z}\right)  ^{1/p}.
\]

2) Let $\gamma_{P}=\gamma_{Q}=\gamma$ and assume that one of the two
alternatives holds:%
\[%
\begin{array}
[c]{c}%
2<\gamma\leqslant d+1\text{ and }p<\left(  2d\right)  /\left(  d+1-\gamma
\right)  ,\\
\text{or}\\
\gamma>d+1~\text{and}~p\leqslant+\infty.
\end{array}
\]
Then there exist constants $\delta>0$ and $c>0$ such that for every
$R\geqslant1$,%
\begin{align*}
&  \left(  \int_{\mathbb{T}^{d}}\left\vert D_{R}\left(  \mathbf{z}\right)
-R^{\left(  d-1\right)  \left(  1-1/\gamma\right)  }\left(  A_{P}\left(
\mathbf{z-}RP\right)  +A_{Q}\left(  \mathbf{z}-RQ\right)  \right)  \right\vert
^{p}d\mathbf{z}\right)  ^{1/p}\ \\
&  \leqslant cR^{\left(  d-1\right)  \left(  1-1/\gamma\right)  -\delta}.
\end{align*}

\end{theorem}

Note that the series that define $A_{P}\left(  \mathbf{z}\right)  $ and
$A_{Q}\left(  \mathbf{z}\right)  $ converge uniformly and absolutely. In
particular these functions are bounded and continuous.

Observe that the asymptotic estimate of point 1) includes the case of a single
flat point, that is $\gamma_{P}>\gamma_{Q}=2$. In point 2) it is not excluded
that for particular values of $P$, $Q$, $G_{P}\left(  0\right)  $,
$G_{Q}\left(  0\right)  $ and $R$, the terms $A_{P}\left(  \mathbf{z-}%
RP\right)  $ and $A_{Q}\left(  \mathbf{z}-RQ\right)  $ may cancel each other
and the discrepancy gets smaller.

\begin{corollary}
\label{Corollary}Under the assumptions in point 2 in the previous theorem
assume furthermore that $G_{P}\left(  0\right)  =G_{Q}\left(  0\right)  $ and
that $\left(  d-1\right)  /\gamma$ is an even integer. Then for every $R$ such
that $R\mathbf{m}_{0}\cdot\left(  P-Q\right)  $ is an integer we have
\[
A_{P}\left(  \mathbf{z-}RP\right)  +A_{Q}\left(  \mathbf{z}-RQ\right)  =0.
\]
In particular with this choice of the parameters%
\[
\left(  \int_{\mathbb{T}^{d}}\left\vert D_{R}\left(  \mathbf{z}\right)
\right\vert ^{p}d\mathbf{z}\right)  ^{1/p}\leqslant cR^{\left(  d-1\right)
\left(  1-1/\gamma\right)  -\delta}.
\]

\end{corollary}

The case $\gamma=2$ is not covered by the above corollary, however observe
that for $\gamma=\left(  d-1\right)  /\left(  2k\right)  =2$, that is
$d\equiv1\left(  \operatorname{mod}\text{ }4\right)  $, and $p<2d/\left(
d-1\right)  $, one formally would obtain%
\[
\liminf_{R\rightarrow+\infty}\left\{  R^{\frac{d-1}{2}}\left(  \int
_{\mathbb{T}^{d}}\left\vert D_{R}\left(  \mathbf{z}\right)  \right\vert
^{p}d\mathbf{z}\right)  ^{1/p}\right\}  =0.
\]
Actually this is true, even if the proof is more delicate. The case of a ball
and $p=2$ has been proved by L. Parnovski and A. Sobolev in \cite{PS}%
.\ Moreover, in \cite{BCGT} it is shown that this phenomenon also occurs for
convex smooth domains with positive Gaussian curvature and $p<2d/\left(
d-1\right)  $ if and only if the domains are symmetric and $d\equiv1\left(
\operatorname{mod}\text{ }4\right)  $.

As remarked by Kendal the above $L^{p}$ estimates for the discrepancy can be
turned into almost everywhere pointwise estimates using a Borel-Cantelli type
argument. See \cite[\S 3]{K} for the proof.

\begin{proposition}
\label{Kendall}Assume that for some $\beta>0$%
\[
\left(  \int_{\mathbb{T}^{d}}\left\vert D_{R}\left(  \mathbf{z}\right)
\right\vert ^{p}\right)  ^{1/p}\leqslant\kappa R^{\beta},
\]
let $\lambda\left(  t\right)  $ be an increasing function and let
$R_{n}\rightarrow+\infty$ such that%
\[
\sum_{n=1}^{+\infty}\lambda\left(  R_{n}\right)  ^{-p}<+\infty.
\]
Then for almost every $\mathbf{z}\in\mathbb{T}^{d}$ there exists $c>0$ such
that%
\[
\left\vert D_{R_{n}}\left(  \mathbf{z}\right)  \right\vert <cR_{n}^{\beta
}\lambda\left(  R_{n}\right)  .
\]

\end{proposition}

If the flat points on the boundary of the domain $B$ have \textquotedblleft
irrational\textquotedblright\ normals then the discrepancy can be smaller than
the one described in the above theorems. In particular, we have the following
result that applies to every convex body, without curvature or smoothness assumption.

\begin{theorem}
\label{Theorem C}Let $B$ be a bounded convex body in $\mathbb{R}^{d}$ and for
$\sigma\in SO\left(  d\right)  $ denote by $D_{R,\sigma}$ the discrepancy
associated to the rotated body $\sigma B$. Then we have the following mixed
norm inequalities.

1) If $1\leqslant p\leqslant2$, we have%
\[
\left(  \int_{SO\left(  d\right)  }\left(  \int_{\mathbb{T}^{d}}\left\vert
D_{R,\sigma}\left(  \mathbf{z}\right)  \right\vert ^{p}d\mathbf{z}\right)
^{2/p}d\sigma\right)  ^{1/2}\leqslant cR^{\frac{d-1}{2}}.
\]

2) If $2\leqslant p<2d/\left(  d-1\right)  $, we have
\[
\left(  \int_{SO\left(  d\right)  }\left(  \int_{\mathbb{T}^{d}}\left\vert
D_{R,\sigma}\left(  \mathbf{z}\right)  \right\vert ^{p}d\mathbf{z}\right)
^{1/\left(  p-1\right)  }d\sigma\right)  ^{\frac{p-1}{p}}\leqslant
cR^{\frac{d-1}{2}}.
\]

\end{theorem}

For the planar case $d=2$ we can state a slightly more precise result.

\begin{theorem}
\label{Theorem D}Let $B$ be a bounded convex body in $\mathbb{R}^{2}$. Assume
that $\partial B$ is smooth with everywhere positive curvature except a single
flat point of order $\gamma>2$. Let $\left(  \alpha,\beta\right)  $ be the
unit outward normal at the flat point and assume the Diophantine property that
for some $\delta<2/\left(  \gamma-2\right)  $ there exists $c>0$ such that for
every $n\in\mathbb{Z}$%
\[
\left\Vert n\frac{\alpha}{\beta}\right\Vert \geqslant\frac{c}{\left\vert
n\right\vert ^{1+\delta}}.
\]
Here $\left\Vert x\right\Vert $ denotes the distance of $x$ from the closest
integer. Then%
\[
\left(  \int_{\mathbb{T}^{d}}\left\vert D_{R}\left(  \mathbf{z}\right)
\right\vert ^{2}d\mathbf{z}\right)  ^{1/2}\leqslant cR^{\frac{1}{2}}.
\]

\end{theorem}

By a classical result of Jarnik (see e.g. \cite[\S 10.3]{F}) the set of real
numbers $\omega$ that are $\left(  2+\delta\right)  $-well approximable, that
is%
\[
\left\Vert n\omega\right\Vert \leqslant n^{1-\left(  2+\delta\right)
}=n^{-1-\delta}%
\]
for infinitely many $n$, has Hausdorff dimension $\frac{2}{2+\delta}$. In
particular the exceptional set in the above theorem, where the discrepancy may
be larger than $R^{1/2}$ has Hausdorff dimension at most $\frac{\gamma
-2}{\gamma-1}$.

\section{Estimates for the Fourier transforms}

The main ingredient in the proof of our results on the discrepancy comes from
suitable estimates of the decay of $\widehat{\chi}_{B}\left(  \mathbf{\zeta
}\right)  $. We start studying a family of oscillatory integrals.

As usual we write $d$-dimensional points through the notation $\mathbf{z}%
=\left(  x,t\right)  $ and $\mathbf{\zeta}=\left(  \xi,s\right)  $ (see the Introduction).

\begin{lemma}
Let $U\subset\mathbb{R}^{d-1}$ be an open ball about the origin of radius $b$,
let $\Phi\in S_{\gamma}\left(  U\right)  $ for some $\gamma>1$, let $\psi$ be
a smooth function supported in $\left\{  0<a\leqslant\left\vert x\right\vert
\leqslant b\right\}  $, for every positive integer $k$ let%
\[
\Phi_{k}\left(  x\right)  =2^{k\gamma}\Phi\left(  2^{-k}x\right)
\]
and let%
\[
I_{k}\left(  \xi,s\right)  =\int_{\mathbb{R}^{d-1}}\left(  \nabla\Phi\left(
2^{-k}x\right)  ,-1\right)  e^{-2\pi i\left(  \xi,s\right)  \cdot\left(
x,\Phi_{k}\left(  x\right)  \right)  }\psi\left(  x\right)  dx.
\]
Then there exist constants $c,c_{1},c_{2}>0$ and, for every $M>0$, a constant
$c_{M}$ such that for every $k\geqslant0$%
\[
\left\vert I_{k}\left(  \xi,s\right)  \right\vert \leqslant\left\{
\begin{array}
[c]{ll}%
c\left(  1+\left\vert s\right\vert +\left\vert \xi\right\vert \right)
^{-\frac{d-1}{2}} & \text{for every }\left(  \xi,s\right) \\
c_{M}\left(  1+\left\vert s\right\vert \right)  ^{-M} & \text{if }\left\vert
\xi\right\vert \leqslant c_{1}\left\vert s\right\vert ,\\
c_{M}\left(  1+\left\vert \xi\right\vert \right)  ^{-M} & \text{if }%
c_{2}\left\vert s\right\vert \leqslant\left\vert \xi\right\vert .
\end{array}
\right.
\]

\end{lemma}

\begin{proof}
The behaviour of the oscillatory integral $I_{k}\left(  \xi,s\right)  $
depends on the points where the amplitude $\psi\left(  x\right)  $ is not zero
and the phase $\left(  \xi,s\right)  \cdot\left(  x,\Phi_{k}\left(  x\right)
\right)  $ is stationary. This happens only when $\left\vert \xi\right\vert
\approx\left\vert s\right\vert $ and in this case, since the phase is non
degenerate, one obtains the classical estimate $c\left\vert \left(
\xi,s\right)  \right\vert ^{-\left(  d-1\right)  /2}$. In all other directions
the oscillatory integral has a fast decay. In particular, when $\left\vert
\xi\right\vert \leqslant c_{1}\left\vert s\right\vert $ one obtains the decay
$c_{M}\left(  1+\left\vert s\right\vert \right)  ^{-M}$, and when $\left\vert
\xi\right\vert >c_{2}\left\vert s\right\vert $, one obtains the decay
$c_{M}\left(  1+\left\vert \xi\right\vert \right)  ^{-M}$. For the sake of
completeness we include the full details of the proof.

By the definition of the class $S_{\gamma}$ we have%
\begin{align}
\left\vert \frac{\partial^{\left\vert \alpha\right\vert }\Phi_{k}}{\partial
x^{\alpha}}\left(  x\right)  \right\vert  &  =2^{k\gamma}2^{-k\left\vert
\alpha\right\vert }\left\vert \frac{\partial^{\left\vert \alpha\right\vert
}\Phi}{\partial x^{\alpha}}\left(  2^{-k}x\right)  \right\vert \leqslant
2^{k\gamma}2^{-k\left\vert \alpha\right\vert }c_{\alpha}\left\vert
2^{-k}x\right\vert ^{\gamma-\left\vert \alpha\right\vert }\label{Der Unif}\\
&  \leqslant c_{\alpha}\left\vert x\right\vert ^{\gamma-\left\vert
\alpha\right\vert }.\nonumber
\end{align}
In particular when $x$ belongs to the support of $\psi\left(  x\right)  $ we
have%
\[
\left\vert \frac{\partial^{\left\vert \alpha\right\vert }\Phi_{k}}{\partial
x^{\alpha}}\left(  x\right)  \right\vert \leqslant c_{\alpha}.
\]
Moreover, the Hessian matrix of $\Phi_{k}\left(  x\right)  $ satisfies%
\[
\operatorname{Hess}\Phi_{k}\left(  x\right)  =2^{k\left(  \gamma-2\right)
}\operatorname{Hess}\Phi\left(  2^{-k}x\right)
\]
and it follows that the eigenvalues $\mu_{j}^{\left(  k\right)  }\left(
x\right)  $ of $\operatorname{Hess}\Phi_{k}\left(  x\right)  $ are related to
the eigenvalues $\mu_{j}\left(  x\right)  $ of $\operatorname{Hess}\Phi\left(
x\right)  $ by the identity%
\[
\mu_{j}^{\left(  k\right)  }\left(  x\right)  =2^{k\left(  \gamma-2\right)
}\mu_{j}\left(  2^{-k}x\right)  .
\]
By (\ref{Eigenvalues})%
\begin{equation}
\mu_{j}^{\left(  k\right)  }\left(  x\right)  =\frac{\mu_{j}\left(
2^{-k}x\right)  }{\left\vert 2^{-k}x\right\vert ^{\gamma-2}}\left\vert
x\right\vert ^{\gamma-2}\geqslant c\left\vert x\right\vert ^{\gamma-2}.
\label{pos eigen}%
\end{equation}
If $x$ belongs to the support of $\psi\left(  x\right)  $ we have
\[
\mu_{j}^{\left(  k\right)  }\left(  x\right)  \geqslant c>0.
\]
Since%
\[
\nabla\Phi\left(  2^{-k}x\right)  =2^{-k\left(  \gamma-1\right)  }\nabla
\Phi_{k}\left(  x\right)
\]
by (\ref{Der Unif}) all the derivatives of $\nabla\Phi\left(  2^{-k}x\right)
$ are uniformly bounded.

The phase in the integral $I_{k}\left(  \xi,s\right)  $ is stationary when%
\[
\nabla\left(  \xi\cdot x+s\Phi_{k}\left(  x\right)  \right)  =\xi+s\nabla
\Phi_{k}\left(  x\right)  =0.
\]
By (\ref{Der Unif}) there exits $c_{2}>0$ such that $\left\vert \nabla\Phi
_{k}\left(  x\right)  \right\vert \leqslant\frac{c_{2}}{2}$ for every $k$. It
follows that for $\left\vert \xi\right\vert \geqslant c_{2}\left\vert
s\right\vert $ we have%
\[
\left\vert \xi+s\nabla\Phi_{k}\left(  x\right)  \right\vert \geqslant
\left\vert \xi\right\vert -\left\vert s\right\vert \left\vert \nabla\Phi
_{k}\left(  x\right)  \right\vert \geqslant\frac{1}{2}\left\vert
\xi\right\vert .
\]
Integrating by parts $M$ times gives (see e.g. Proposition 4, p. 341, in
\cite{Stein})%
\[
\left\vert I_{k}\left(  \xi,s\right)  \right\vert \leqslant c_{M}\left(
1+\left\vert \xi\right\vert \right)  ^{-M}.
\]
Let now $\left\vert \xi\right\vert \leqslant c_{1}\left\vert s\right\vert $
where $c_{1}$ is a constant which will be determined later on. Let us consider
the function
\[
F\left(  t\right)  =\nabla\Phi_{k}\left(  tx\right)  \cdot x
\]
with $t\in\left[  0,1\right]  $. Then, for $t\in\left(  0,1\right]  $%
\[
F^{\prime}\left(  t\right)  =x^{T}\operatorname{Hess}\Phi_{k}\left(
tx\right)  x
\]
and by (\ref{pos eigen}) the eigenvalues of $\operatorname{Hess}\Phi
_{k}\left(  tx\right)  $ are bounded from below by $t^{\gamma-2}\left\vert
x\right\vert ^{\gamma-2}$. Then%
\[
F\left(  1\right)  \geqslant\int_{0}^{1}x^{T}\operatorname{Hess}\Phi
_{k}\left(  tx\right)  xdt\geqslant c\int_{0}^{1}t^{\gamma-2}\left\vert
x\right\vert ^{\gamma-2}\left\vert x\right\vert ^{2}dt\geqslant c\left\vert
x\right\vert ^{\gamma}%
\]
and therefore%
\[
\left\vert \nabla\Phi_{k}\left(  x\right)  \right\vert \geqslant\frac
{\nabla\Phi_{k}\left(  x\right)  \cdot x}{\left\vert x\right\vert }\geqslant
c\left\vert x\right\vert ^{\gamma-1}\geqslant2c_{1}>0.
\]
It follows that%
\[
\left\vert \xi+s\nabla\Phi_{k}\left(  x\right)  \right\vert \geqslant
\left\vert s\right\vert \left\vert \nabla\Phi_{k}\left(  x\right)  \right\vert
-\left\vert \xi\right\vert \geqslant c_{1}\left\vert s\right\vert .
\]
Integrating by parts $M$ times gives%
\[
\left\vert I_{k}\left(  \xi,s\right)  \right\vert \leqslant c_{M}\left(
1+\left\vert s\right\vert \right)  ^{-M}.
\]
Finally, for every $\left(  \xi,s\right)  $, Theorem 1, p. 348, in
\cite{Stein} gives%
\[
\left\vert I_{k}\left(  \xi,s\right)  \right\vert \leqslant c\left(
1+\left\vert \xi\right\vert +\left\vert s\right\vert \right)  ^{-\frac{d-1}%
{2}}.
\]

\end{proof}

\begin{proposition}
\label{Prop Stima chi}Let $\gamma>1$ and let $B$ be a bounded convex body in
$\mathbb{R}^{d}$ with everywhere positive Gaussian curvature with the
exception of a single flat point of order $\gamma$. Let $\Theta$ be the
outward unit normal to $\partial B$ at the flat point and for every
$\mathbf{\zeta}\in\mathbb{R}^{d}$ write $\mathbf{\zeta}=\mathbf{\xi}+s\Theta$,
with $s=\mathbf{\zeta}\cdot\Theta$ and $\mathbf{\xi}\cdot\Theta=0$. Then, if
$1<\gamma\leqslant2$%
\begin{equation}
\left\vert \widehat{\chi}_{B}\left(  \mathbf{\zeta}\right)  \right\vert
\leqslant c\left\vert \mathbf{\zeta}\right\vert ^{-\frac{d+1}{2}}.
\label{Stima Chi ottimale}%
\end{equation}
If $\gamma>2$ the following three upper bounds hold:%
\begin{equation}
\left\vert \widehat{\chi}_{B}\left(  \mathbf{\zeta}\right)  \right\vert
\leqslant\left\{
\begin{array}
[c]{l}%
c\left\vert s\right\vert ^{-1-\frac{d-1}{\gamma}},\\
c\left\vert \mathbf{\xi}\right\vert ^{-\left(  d-1\right)  \frac{\gamma
-2}{2\left(  \gamma-1\right)  }}\left\vert s\right\vert ^{-\frac{d-1}{2\left(
\gamma-1\right)  }-1},\\
c\left\vert \mathbf{\xi}\right\vert ^{-\frac{d+1}{2}}.
\end{array}
\right.  \label{Stima gamma>2}%
\end{equation}

\end{proposition}

The particular case where the boundary in a neighborhood of the flat point has
equation $t=\left\vert x\right\vert ^{\gamma}$ with $\gamma>2$ and $d=2$ has
been already considered in \cite{BRT}. The same case with $\gamma>2$ and $d>2$
has been considered in \cite{BGT}, but we acknowledge that the proof of the
rate of decay in the horizontal directions was not correctly justified.

\begin{proof}
Choose a smooth function $\eta\left(  \mathbf{z}\right)  $ supported in a
neighborhood of the flat point and such that $\eta\left(  \mathbf{z}\right)
=1$ in a smaller neighborhood. For every $\mathbf{z}\in\partial B$ let
$\mathbf{\nu}\left(  \mathbf{z}\right)  $ be its outward unit normal. Applying
the divergence theorem we decompose the Fourier transform as%
\begin{align}
\widehat{\chi}_{B}\left(  \mathbf{\zeta}\right)   &  =\int_{B}e^{-2\pi
i\mathbf{\zeta}\cdot\mathbf{z}}d\mathbf{z}=\frac{-1}{4\pi^{2}\left\vert
\mathbf{\zeta}\right\vert ^{2}}\int_{\partial B}\nabla\left(  e^{-2\pi
i\mathbf{\zeta}\cdot\mathbf{z}}\right)  \cdot\mathbf{\nu}\left(
\mathbf{z}\right)  d\sigma\left(  \mathbf{z}\right) \nonumber\\
&  =\frac{-1}{2\pi i\left\vert \mathbf{\zeta}\right\vert ^{2}}\int_{\partial
B}\mathbf{\zeta}\cdot\mathbf{\nu}\left(  \mathbf{z}\right)  e^{-2\pi
i\mathbf{\zeta}\cdot\mathbf{z}}d\sigma\left(  \mathbf{z}\right) \nonumber\\
&  =\frac{-1}{2\pi i\left\vert \mathbf{\zeta}\right\vert ^{2}}\int_{\partial
B}\mathbf{\zeta}\cdot\mathbf{\nu}\left(  \mathbf{z}\right)  e^{-2\pi
i\mathbf{\zeta}\cdot\mathbf{z}}\eta\left(  \mathbf{z}\right)  d\sigma\left(
\mathbf{z}\right) \label{Def I_1 e I_2}\\
&  +\frac{-1}{2\pi i\left\vert \mathbf{\zeta}\right\vert ^{2}}\int_{\partial
B}\mathbf{\zeta}\cdot\mathbf{\nu}\left(  \mathbf{z}\right)  e^{-2\pi
i\mathbf{\zeta}\cdot\mathbf{z}}\left[  1-\eta\left(  \mathbf{z}\right)
\right]  d\sigma\left(  \mathbf{z}\right) \nonumber\\
&  =K_{1}\left(  \mathbf{\zeta}\right)  +K_{2}\left(  \mathbf{\zeta}\right)
.\nonumber
\end{align}
Since in the support of the function $1-\eta\left(  \mathbf{z}\right)  $ the
Gaussian curvature is bounded away from zero, the method of stationary phase
gives the classical estimate (see Theorem 1, p. 348, in \cite{Stein})%
\begin{equation}
\left\vert K_{2}\left(  \mathbf{\zeta}\right)  \right\vert \leqslant
c\left\vert \mathbf{\zeta}\right\vert ^{-\frac{d+1}{2}}. \label{Stima I_2}%
\end{equation}
By a suitable choice of coordinates we can assume that $\mathbf{z}=\left(
x,t\right)  $, the flat point is the point $\left(  0,0\right)  $, its outward
normal is $\left(  0,-1\right)  $ and that the relevant part of the surface
$\partial B$ is described by the equation $t=\Phi\left(  x\right)  $ with
$\Phi\in S_{\gamma}$. Hence%
\[
\mathbf{\nu}\left(  x,\Phi\left(  x\right)  \right)  =\frac{\left(  \nabla
\Phi\left(  x\right)  ,-1\right)  }{\sqrt{1+\left\vert \nabla\Phi\left(
x\right)  \right\vert ^{2}}}.
\]
Write $\varphi\left(  x\right)  =\eta\left(  x,\Phi\left(  x\right)  \right)
$ and $\psi\left(  x\right)  =\varphi\left(  x\right)  -\varphi\left(
2x\right)  $ so that for every $x\neq0$%
\[
\varphi\left(  x\right)  =\sum_{k=0}^{+\infty}\psi\left(  2^{k}x\right)  .
\]
Observe that $\varphi\left(  x\right)  $ is smooth and a suitable choice of
$\eta\left(  \mathbf{z}\right)  $ guarantees $\varphi\left(  x\right)  =1$ if
$\left\vert x\right\vert \leqslant\varepsilon/2$ and $\varphi\left(  x\right)
=0$ if $\left\vert x\right\vert \geqslant\varepsilon$ for some $\varepsilon
>0$. With the above choice of coordinate we can also write $\mathbf{\zeta
=}\left(  \xi,-s\right)  $ so that, following the notation of the previous
lemma, we have%
\begin{align}
&  K_{1}\left(  \mathbf{\zeta}\right) \nonumber\\
&  =\frac{-1}{2\pi i\left\vert \mathbf{\zeta}\right\vert ^{2}}\int
_{\mathbb{R}^{d-1}}\mathbf{\zeta}\cdot\frac{\left(  \nabla\Phi\left(
x\right)  ,-1\right)  }{\sqrt{1+\left\vert \nabla\Phi\left(  x\right)
\right\vert ^{2}}}e^{-2\pi i\mathbf{\zeta}\cdot\left(  x,\Phi\left(  x\right)
\right)  }\varphi\left(  x\right)  \sqrt{1+\left\vert \nabla\Phi\left(
x\right)  \right\vert ^{2}}dx\nonumber\\
&  =\frac{-\mathbf{\zeta}}{2\pi i\left\vert \mathbf{\zeta}\right\vert ^{2}%
}\cdot\int_{\mathbb{R}^{d-1}}\left(  \nabla\Phi\left(  x\right)  ,-1\right)
e^{-2\pi i\mathbf{\zeta}\cdot\left(  x,\Phi\left(  x\right)  \right)  }%
\varphi\left(  x\right)  dx\label{Stima I_1}\\
&  =\frac{-\mathbf{\zeta}}{2\pi i\left\vert \mathbf{\zeta}\right\vert ^{2}%
}\cdot\sum_{k=0}^{+\infty}\int_{\mathbb{R}^{d-1}}\left(  \nabla\Phi\left(
x\right)  ,-1\right)  e^{-2\pi i\mathbf{\zeta}\cdot\left(  x,\Phi\left(
x\right)  \right)  }\psi\left(  2^{k}x\right)  dx\nonumber\\
&  =\sum_{k=0}^{+\infty}\dfrac{-2^{-k\left(  d-1\right)  }\mathbf{\zeta}}{2\pi
i\left\vert \mathbf{\zeta}\right\vert ^{2}}\nonumber\\
&  \ \ \ \ \ \ \cdot\int_{\mathbb{R}^{d-1}}\left(  \nabla\Phi\left(
2^{-k}y\right)  ,-1\right)  e^{-2\pi i\left(  2^{-k}\xi,-2^{-k\gamma}s\right)
\cdot\left(  y,2^{k\gamma}\Phi\left(  2^{-k}y\right)  \right)  }\psi\left(
y\right)  dy\nonumber\\
&  =\sum_{k=0}^{+\infty}2^{-k\left(  d-1\right)  }\frac{-\mathbf{\zeta}}{2\pi
i\left\vert \mathbf{\zeta}\right\vert ^{2}}\cdot I_{k}\left(  2^{-k}%
\xi,-2^{-k\gamma}s\right)  .\nonumber
\end{align}
By the previous lemma%
\[
\left\vert I_{k}\left(  2^{-k}\xi,-2^{-k\gamma}s\right)  \right\vert \leqslant
c\left(  1+\left\vert 2^{-k\gamma}s\right\vert +\left\vert 2^{-k}%
\xi\right\vert \right)  ^{-\frac{d-1}{2}}.
\]
Hence,%
\[
\left\vert K_{1}\left(  \mathbf{\zeta}\right)  \right\vert \leqslant\frac
{c}{\left\vert \left(  \xi,s\right)  \right\vert }\sum_{k=0}^{+\infty
}2^{-k\left(  d-1\right)  }\left(  1+\left\vert 2^{-k\gamma}s\right\vert
+\left\vert 2^{-k}\xi\right\vert \right)  ^{-\frac{d-1}{2}}.
\]
In particular, for every $\left(  \xi,s\right)  \in\mathbb{R}^{d}$ we have%
\[
\left\vert K_{1}\left(  \mathbf{\zeta}\right)  \right\vert \leqslant\frac
{c}{\left\vert \xi\right\vert }\sum_{k=0}^{+\infty}2^{-k\left(  d-1\right)
}\left\vert 2^{-k}\xi\right\vert ^{-\frac{d-1}{2}}\leqslant c\left\vert
\xi\right\vert ^{-\frac{d+1}{2}}.
\]
Assume now $\gamma\neq2$. Our second estimate for $K_{1}\left(  \mathbf{\zeta
}\right)  $ is as follows. For every $\left(  \xi,s\right)  \in\mathbb{R}^{d}$
we have%
\begin{align*}
\left\vert K_{1}\left(  \mathbf{\zeta}\right)  \right\vert  &  \leqslant
\frac{c}{\left\vert s\right\vert }\sum_{k=0}^{+\infty}2^{-k\left(  d-1\right)
}\left(  1+\left\vert 2^{-k\gamma}s\right\vert \right)  ^{-\frac{d-1}{2}}\\
&  \leqslant\frac{c}{\left\vert s\right\vert }\sum_{2^{k}\geqslant\left\vert
s\right\vert ^{1/\gamma}}2^{-k\left(  d-1\right)  }+\frac{c}{\left\vert
s\right\vert }\sum_{2^{k}<\left\vert s\right\vert ^{1/\gamma}}2^{-k\left(
d-1\right)  }\left\vert 2^{-k\gamma}s\right\vert ^{-\frac{d-1}{2}}\\
&  \leqslant c\left\vert s\right\vert ^{-1-\frac{d-1}{\gamma}}+c\left\vert
s\right\vert ^{-1-\frac{d-1}{2}}\sum_{2^{k}<\left\vert s\right\vert
^{1/\gamma}}2^{k\left(  d-1\right)  \left(  \frac{\gamma}{2}-1\right)  }\\
&  \leqslant\left\{
\begin{array}
[c]{ll}%
c\left\vert s\right\vert ^{-\frac{d+1}{2}} & \gamma<2,\\
c\left\vert s\right\vert ^{-1-\frac{d-1}{\gamma}} & \gamma>2.
\end{array}
\right.
\end{align*}
Note that when $\gamma=2$ the previous computation gives $c\left\vert
s\right\vert ^{-\frac{d+1}{2}}\log\left(  2+\left\vert s\right\vert \right)
$. However when $\Phi\left(  y\right)  $ is smooth it is well known that the
correct estimate is $c\left\vert s\right\vert ^{-\frac{d+1}{2}}$. With a more
careful analysis we show that this is the case also in our setting, even if we
do not assume smoothness at the flat point. Indeed notice that condition
(\ref{Deriv-Phi}) allows higher derivatives to blow up at the flat point. Let
$\gamma=2$. Then%
\begin{equation}
\left\vert K_{1}\left(  \mathbf{\zeta}\right)  \right\vert \leqslant\frac
{c}{\left\vert \left(  \xi,s\right)  \right\vert }\sum_{k=0}^{+\infty
}2^{-k\left(  d-1\right)  }\left\vert I_{k}\left(  2^{-k}\xi,-2^{-k\gamma
}s\right)  \right\vert \label{trePezzi}%
\end{equation}
By the previous lemma we have%
\[
\left\vert I_{k}\left(  2^{-k}\xi,-2^{-2k}s\right)  \right\vert \leqslant
\left\{
\begin{array}
[c]{ll}%
c\left(  1+\left\vert 2^{-2k}s\right\vert +\left\vert 2^{-k}\xi\right\vert
\right)  ^{-\frac{d-1}{2}} & \text{for every }\left(  \xi,s\right) \\
c_{M}\left(  1+\left\vert 2^{-2k}s\right\vert \right)  ^{-M} & \text{if }%
2^{k}\leqslant c_{1}\frac{\left\vert s\right\vert }{\left\vert \xi\right\vert
},\\
c_{M}\left(  1+\left\vert 2^{-k}\xi\right\vert \right)  ^{-M} & \text{if
}c_{2}\frac{\left\vert s\right\vert }{\left\vert \xi\right\vert }%
\leqslant2^{k}.
\end{array}
\right.
\]
so that%
\begin{align*}
&  \sum_{k=0}^{+\infty}2^{-k\left(  d-1\right)  }\left\vert I_{k}\left(
2^{-k}\xi,-2^{-2k}s\right)  \right\vert \\
&  \leqslant c\sum_{2^{k}\leqslant c_{1}\left\vert s\right\vert /\left\vert
\xi\right\vert }2^{-k\left(  d-1\right)  }\left(  1+\left\vert 2^{-2k}%
s\right\vert \right)  ^{-M}+c\sum_{c_{2}\left\vert s\right\vert /\left\vert
\xi\right\vert \leqslant2^{k}}2^{-k\left(  d-1\right)  }\left(  1+\left\vert
2^{-k}\xi\right\vert \right)  ^{-M}\\
&  +c\sum_{c_{1}\left\vert s\right\vert /\left\vert \xi\right\vert
<2^{k}<c_{2}\left\vert s\right\vert /\left\vert \xi\right\vert }2^{-k\left(
d-1\right)  }\left(  1+\left\vert 2^{-2k}s\right\vert +\left\vert 2^{-k}%
\xi\right\vert \right)  ^{-\frac{d-1}{2}}\\
&  =S_{1}+S_{2}+S_{3}.
\end{align*}
We have%
\begin{align*}
S_{1}  &  \leqslant c\sum_{k=0}^{+\infty}2^{-k\left(  d-1\right)  }\left(
1+\left\vert 2^{-2k}s\right\vert \right)  ^{-M}\\
&  \leqslant c\sum_{\left\vert s\right\vert ^{1/2}\leqslant2^{k}}2^{-k\left(
d-1\right)  }+c\left\vert s\right\vert ^{-M}\sum_{\left\vert s\right\vert
^{1/2}>2^{k}}2^{2kM}2^{-k\left(  d-1\right)  }\\
&  \leqslant c\left\vert s\right\vert ^{-\frac{d-1}{2}}%
\end{align*}
and%
\begin{align*}
S_{2}  &  \leqslant c\sum_{c_{2}\left\vert s\right\vert /\left\vert
\xi\right\vert \leqslant2^{k}}2^{-k\left(  d-1\right)  }\left(  1+\left\vert
2^{-k}\xi\right\vert \right)  ^{-M}\\
&  \leqslant c\sum_{\max\left(  c_{2}\left\vert s\right\vert /\left\vert
\xi\right\vert ,\left\vert \xi\right\vert \right)  \leqslant2^{k}}2^{-k\left(
d-1\right)  }+c\left\vert \xi\right\vert ^{-M}\sum_{c_{2}\left\vert
s\right\vert /\left\vert \xi\right\vert \leqslant2^{k}<\left\vert
\xi\right\vert }2^{Mk-k\left(  d-1\right)  }.
\end{align*}
Observe that%
\begin{align*}
\sum_{\max\left(  c_{2}\left\vert s\right\vert /\left\vert \xi\right\vert
,\left\vert \xi\right\vert \right)  \leqslant2^{k}}2^{-k\left(  d-1\right)  }
&  \leqslant c\min\left(  \left(  \frac{\left\vert \xi\right\vert }{\left\vert
s\right\vert }\right)  ^{d-1},\left\vert \xi\right\vert ^{-\left(  d-1\right)
}\right) \\
&  \leqslant c\left\vert s\right\vert ^{-\frac{d-1}{2}}\min\left(
\frac{\left\vert \xi\right\vert ^{d-1}}{\left\vert s\right\vert ^{\frac
{d-1}{2}}},\frac{\left\vert s\right\vert ^{\frac{d-1}{2}}}{\left\vert
\xi\right\vert ^{d-1}}\right)  \leqslant c\left\vert s\right\vert
^{-\frac{d-1}{2}}.
\end{align*}
Also, the series%
\[
\left\vert \xi\right\vert ^{-M}\sum_{c_{2}\left\vert s\right\vert /\left\vert
\xi\right\vert \leqslant2^{k}<\left\vert \xi\right\vert }2^{Mk-k\left(
d-1\right)  }%
\]
is non void only if $c_{2}\frac{\left\vert s\right\vert }{\left\vert
\xi\right\vert }<\left\vert \xi\right\vert $, that is $c\left\vert
s\right\vert ^{1/2}<\left\vert \xi\right\vert $. In this case we have%
\begin{align*}
c\left\vert \xi\right\vert ^{-M}\sum_{c_{2}\left\vert s\right\vert /\left\vert
\xi\right\vert \leqslant2^{k}<\left\vert \xi\right\vert }2^{Mk-k\left(
d-1\right)  }  &  \leqslant c\left\vert \xi\right\vert ^{-M}\sum_{\left\vert
\xi\right\vert >2^{k}}2^{Mk-k\left(  d-1\right)  }\leqslant c\left\vert
\xi\right\vert ^{-\left(  d-1\right)  }\\
&  \leqslant c\left\vert s\right\vert ^{-\frac{d-1}{2}}.
\end{align*}
We now turm to $S_{3}$ that contains a sum with a finite number of terms. We
have,%
\[
S_{3}\leqslant c\left\vert s\right\vert ^{-\frac{d-1}{2}}\sum_{c_{1}\left\vert
s\right\vert /\left\vert \xi\right\vert <2^{k}<c_{2}\left\vert s\right\vert
/\left\vert \xi\right\vert }1\leqslant c\left\vert s\right\vert ^{-\frac
{d-1}{2}}.
\]
Substituting into (\ref{trePezzi}) gives the estimate%
\[
\left\vert \widehat{\chi}_{B}\left(  \mathbf{\zeta}\right)  \right\vert
\leqslant c\left\vert s\right\vert ^{-\frac{d+1}{2}}%
\]
when $\gamma=2$. It remains to prove the second row in (\ref{Stima gamma>2}).
We have%
\begin{align*}
&  \left\vert K_{1}\left(  \mathbf{\zeta}\right)  \right\vert \\
&  \leqslant\frac{c}{\left\vert s\right\vert }\sum_{k=0}^{+\infty}2^{-k\left(
d-1\right)  }\left(  \left\vert 2^{-k\gamma}s\right\vert +\left\vert 2^{-k}%
\xi\right\vert \right)  ^{-\frac{d-1}{2}}\\
&  \leqslant\frac{c}{\left\vert s\right\vert }\left(  \left\vert s\right\vert
^{-\frac{d-1}{2}}\sum_{2^{k}\leqslant\left(  \left\vert s\right\vert
/\left\vert \xi\right\vert \right)  ^{\frac{1}{\gamma-1}}}2^{k\left(
d-1\right)  \left(  \frac{\gamma}{2}-1\right)  }+\left\vert \xi\right\vert
^{-\frac{d-1}{2}}\sum_{2^{k}>\left(  \left\vert s\right\vert /\left\vert
\xi\right\vert \right)  ^{\frac{1}{\gamma-1}}}2^{-k\frac{d-1}{2}}\right) \\
&  \leqslant c\left\vert \xi\right\vert ^{-\frac{\left(  d-1\right)  \left(
\gamma-2\right)  }{2\left(  \gamma-1\right)  }}\left\vert s\right\vert
^{-\frac{d-1}{2\left(  \gamma-1\right)  }-1}.
\end{align*}

\end{proof}

\begin{remark}
Let $\mathbf{z}\in\partial B,$ let $T_{\mathbf{z}}$ be the tangent hyperplane
to $\partial B$ in $\mathbf{z}$ and let%
\[
S\left(  \mathbf{z},\delta\right)  =\left\{  \mathbf{w}\in\partial
B:\operatorname{dist}\left(  \mathbf{w},T_{\mathbf{z}}\right)  <\delta
\right\}  .
\]
In \cite{BNW} it is proved that when the boundary of $B$ is smooth and of
finite type (every one dimensional tangent line to $\partial B$ makes finite
order of contact with $\partial B$), then%
\[
\left\vert \widehat{\chi}_{B}\left(  \mathbf{\zeta}\right)  \right\vert
\leqslant c\left\vert \mathbf{\zeta}\right\vert ^{-1}\left[  \sigma\left(
S\left(  \mathbf{z}^{+},\left\vert \mathbf{\zeta}\right\vert ^{-1}\right)
\right)  +\sigma\left(  S\left(  \mathbf{z}^{-},\left\vert \mathbf{\zeta
}\right\vert ^{-1}\right)  \right)  \right]
\]
where $\mathbf{z}^{+}$ and $\mathbf{z}^{-}$ are the two points on $\partial B$
with outer normal parallel to $\mathbf{\zeta}$ and $\sigma$ is the surface
measure. In our case $\partial B$ is not necessarily smooth, but the above
result in fact holds. Indeed, let $B$ as in Proposition \ref{Prop Stima chi},
and choose coordinates such that $\mathbf{z}=\left(  x,t\right)  $, the flat
point is the point $\left(  0,0\right)  $, its outward normal is $\left(
0,-1\right)  $ and that the relevant part of the surface $\partial B$ is
described by the equation $t=\Phi\left(  x\right)  $ with $\Phi\in S_{\gamma}%
$. Fix $\mathbf{z}=\left(  x,t\right)  \in\partial B$. Elementary geometric
observations lead to%
\begin{align}
\sigma\left(  S\left(  \mathbf{z},\delta\right)  \right)   &  \geqslant
c\sigma\left(  S\left(  \mathbf{0},c_{1}\delta\right)  \right)  \approx\left(
\delta^{1/\gamma}\right)  ^{d-1}~~~\text{for }\delta\geqslant c\left\vert
x\right\vert ^{\gamma},\label{1}\\
\sigma\left(  S\left(  \mathbf{z},\delta\right)  \right)   &  \geqslant
c\left(  \delta\left(  \left\vert x\right\vert ^{\gamma-2}\right)
^{-1}\right)  ^{\frac{d-1}{2}}~~~\text{for }\delta\leqslant c\left\vert
x\right\vert ^{\gamma}. \label{2}%
\end{align}
The unit normal to $\partial B$ in $\left(  x,\Phi\left(  x\right)  \right)  $
is $\frac{\left(  \nabla\Phi\left(  x\right)  ,-1\right)  }{\sqrt{\left\vert
\nabla\Phi\left(  x\right)  \right\vert ^{2}+1}}$. It follows that for a given
$\mathbf{\zeta}=\left(  \xi,s\right)  $ the point $\left(  x,\Phi\left(
x\right)  \right)  $ in $\partial B$ with normal in the direcion
$\mathbf{\zeta}$ satisfies $\left\vert \xi\right\vert /\left\vert s\right\vert
=\left\vert \nabla\Phi\left(  x\right)  \right\vert \approx\left\vert
x\right\vert ^{\gamma-1}$.

a) If $\left\vert s\right\vert \geqslant\left\vert \xi\right\vert ^{\gamma}$
then $\left\vert s\right\vert ^{1-\gamma}\geqslant\left(  \left\vert
\xi\right\vert /\left\vert s\right\vert \right)  ^{\gamma}\approx\left\vert
x\right\vert ^{\left(  \gamma-1\right)  \gamma}$. Hence%
\[
\frac{1}{\left\vert \mathbf{\zeta}\right\vert }\approx\frac{1}{\left\vert
s\right\vert }\geqslant c\left\vert x\right\vert ^{\gamma}%
\]
so that, by (\ref{1}),%
\[
\sigma\left(  S\left(  \mathbf{z},\left\vert \mathbf{\zeta}\right\vert
^{-1}\right)  \right)  \geqslant c\sigma\left(  S\left(  \mathbf{0}%
,c_{1}\left\vert \mathbf{\zeta}\right\vert ^{-1}\right)  \right)
\approx\left\vert \mathbf{\zeta}\right\vert ^{-\left(  d-1\right)  /\gamma
}\approx\left\vert s\right\vert ^{-\left(  d-1\right)  /\gamma}.
\]
b) If $\left\vert \xi\right\vert \leqslant\left\vert s\right\vert
\leqslant\left\vert \xi\right\vert ^{\gamma}$, then as before%
\[
\frac{1}{\left\vert \mathbf{\zeta}\right\vert }\approx\frac{1}{\left\vert
s\right\vert }\leqslant c\left\vert x\right\vert ^{\gamma}%
\]
and by (\ref{2})
\[
\sigma\left(  S\left(  \mathbf{z},\left\vert \mathbf{\zeta}\right\vert
^{-1}\right)  \right)  \geqslant c\left(  \frac{1}{\left\vert s\right\vert
}\left(  \frac{\left\vert \xi\right\vert }{\left\vert s\right\vert }\right)
^{\frac{2-\gamma}{\gamma-1}}\right)  ^{\frac{d-1}{2}}=c\left\vert
\xi\right\vert ^{-\left(  d-1\right)  \frac{\gamma-2}{2\left(  \gamma
-1\right)  }}\left\vert s\right\vert ^{-\frac{d-1}{2\left(  \gamma-1\right)
}}%
\]
c) If $\left\vert \xi\right\vert \geqslant\left\vert s\right\vert $, then
$\left\vert x\right\vert \approx1$and by (\ref{2})%
\[
\sigma\left(  S\left(  \mathbf{z},\left\vert \mathbf{\zeta}\right\vert
^{-1}\right)  \right)  \geqslant\left\vert \xi\right\vert ^{-\frac{d-1}{2}}.
\]
The geometric estimate
\[
\left\vert \widehat{\chi}_{B}\left(  \mathbf{\zeta}\right)  \right\vert
\leqslant c\left\vert \mathbf{\zeta}\right\vert ^{-1}\sigma\left(  S\left(
\mathbf{z},\left\vert \mathbf{\zeta}\right\vert ^{-1}\right)  \right)  .
\]
now follows from Proposition \ref{Prop Stima chi}.
\end{remark}

As said before, the above proposition is the main ingredient in the estimate
of the discrepancy associated to the convex body $B$. In particular it follows
that the directions where the Fourier transform has the slowest rate of decay
play a relevant role in the estimates of the discrepancy.

Actually the Fourier transform in a given direction depends on the two points
in $\partial B$ have normals in that direction. The interplay between the
contribution of these points is exploited in the following proposition.

\begin{proposition}
\label{Prop Sezioni}Let $B$ be a bounded convex body in $\mathbb{R}^{d}$.
Assume that $\partial B$ is smooth with everywhere positive Gaussian curvature
except at most at two points $P$ and $Q$ which are flat of order $\gamma_{P}$
and $\gamma_{Q}$ respectively and have outward unit normals $-\Theta$ and
$\Theta$. Let%
\[
S\left(  t\right)  =\left\vert \left\{  \mathbf{z}\in B:\mathbf{z}\cdot
\Theta=t\right\}  \right\vert
\]
be $\left(  d-1\right)  $-dimensional measures of the slices of $B$ that are
orthogonal to $\Theta$. The function $S\left(  t\right)  $ is supported in
$P\cdot\Theta\leqslant t\leqslant Q\cdot\Theta$ and is known to be smooth in
$P\cdot\Theta<t<Q\cdot\Theta$. Assume that there exist two smooth functions
$G_{P}\left(  r\right)  $ and $G_{Q}\left(  r\right)  $ with $G_{P}\left(
0\right)  \neq0$ and $G_{Q}\left(  0\right)  \neq0$ such that, for
$u\geqslant0$ sufficiently small%
\[
S\left(  P\cdot\Theta+u\right)  =u^{\frac{d-1}{\gamma_{P}}}G_{P}\left(
u^{1/\gamma_{P}}\right)
\]
and%
\[
S\left(  Q\cdot\Theta-u\right)  =u^{\frac{d-1}{\gamma_{Q}}}G_{Q}\left(
u^{1/\gamma_{Q}}\right)  .
\]
Then, as $\left\vert s\right\vert \rightarrow+\infty$,%
\begin{align*}
\widehat{\chi}_{B}\left(  s\Theta\right)   &  =e^{-2\pi is\Theta\cdot P}%
G_{P}\left(  0\right)  \frac{\Gamma\left(  \frac{d-1}{\gamma_{P}}+1\right)
}{\left(  2\pi\right)  ^{\frac{d-1}{\gamma_{P}}+1}}e^{-i\frac{\pi}{2}\left(
\frac{d-1}{\gamma_{P}}+1\right)  \operatorname{sgn}\left(  s\right)
}\left\vert s\right\vert ^{-1-\frac{d-1}{\gamma_{P}}}\\
&  +e^{-2\pi is\Theta\cdot Q}G_{Q}\left(  0\right)  \frac{\Gamma\left(
\frac{d-1}{\gamma_{Q}}+1\right)  }{\left(  2\pi\right)  ^{\frac{d-1}%
{\gamma_{Q}}+1}}e^{i\frac{\pi}{2}\left(  \frac{d-1}{\gamma_{Q}}+1\right)
\operatorname{sgn}\left(  s\right)  }\left\vert s\right\vert ^{-1-\frac
{d-1}{\gamma_{Q}}}\\
&  +O\left(  \left\vert s\right\vert ^{-1-\frac{d}{\max\left(  \gamma
_{P},\gamma_{Q}\right)  }}\right)  .
\end{align*}

\end{proposition}

Observe that when in a neighborhood of the points $P$ and $Q$ the boundary of
$B$ is smooth with positive Gaussian curvature $K\left(  P\right)  $ and
$K\left(  Q\right)  $ then we have $\gamma_{P}=\gamma_{Q}=\gamma=2$,
$G_{P}\left(  0\right)  =\frac{\left(  2\pi\right)  ^{\left(  d-1\right)  /2}%
}{\Gamma\left(  \frac{d+1}{2}\right)  }K^{-1/2}\left(  P\right)  $, and
$G_{Q}\left(  0\right)  =\frac{\left(  2\pi\right)  ^{\left(  d-1\right)  /2}%
}{\Gamma\left(  \frac{d+1}{2}\right)  }K^{-1/2}\left(  Q\right)  $. Hence we
obtain the classical formula%
\begin{align*}
\widehat{\chi}_{B}\left(  s\Theta\right)   &  =\frac{1}{2\pi}e^{-2\pi
is\Theta\cdot P}K^{-1/2}\left(  P\right)  e^{-i\frac{\pi}{2}\left(  \frac
{d-1}{2}+1\right)  \operatorname{sgn}\left(  s\right)  }\left\vert
s\right\vert ^{-\frac{d+1}{2}}\\
&  +\frac{1}{2\pi}e^{-2\pi is\Theta\cdot Q}K^{-1/2}\left(  Q\right)
e^{i\frac{\pi}{2}\left(  \frac{d-1}{2}+1\right)  \operatorname{sgn}\left(
s\right)  }\left\vert s\right\vert ^{-\frac{d+1}{2}}\\
&  +O\left(  \left\vert s\right\vert ^{-\frac{d+2}{2}}\right)  .
\end{align*}
See \cite{He} and \cite{Hl}. See also \cite[Corollary 7.7.15]{Ho}.

\begin{proof}
Let $\eta\left(  t\right)  $ be a smooth cutoff function with $\eta\left(
t\right)  =1$ if $\left\vert t\right\vert \leqslant\varepsilon$ and
$\eta\left(  t\right)  =0$ if $\left\vert t\right\vert \geqslant2\varepsilon$
with $\varepsilon$ small. Since $S\left(  t\right)  $ is smooth inside
$P\cdot\Theta<t<Q\cdot\Theta$, see \cite{Bi}, for every $N>0$ we have%
\begin{align*}
\widehat{\chi}_{B}\left(  s\Theta\right)   &  =\int_{\mathbb{R}^{d}}\chi
_{B}\left(  \mathbf{z}\right)  e^{-2\pi i\mathbf{z}\cdot s\Theta}%
d\mathbf{z}=\int_{-\infty}^{+\infty}\left(  \int_{\left\{  \mathbf{z}%
\cdot\Theta=t\right\}  }\chi_{B}\left(  \mathbf{z}\right)  d\mathbf{z}\right)
e^{-2\pi ist}dt\\
&  =\int_{P\cdot\Theta}^{Q\cdot\Theta}S\left(  t\right)  e^{-2\pi ist}dt\\
&  =\int_{0}^{+\infty}\eta\left(  u\right)  S\left(  P\cdot\Theta+u\right)
e^{-2\pi is\left(  P\cdot\Theta+u\right)  }du\\
&  +\int_{0}^{+\infty}\eta\left(  u\right)  S\left(  Q\cdot\Theta-u\right)
e^{-2\pi is\left(  Q\cdot\Theta-u\right)  }du+O\left(  \left\vert s\right\vert
^{-N}\right) \\
&  =e^{-2\pi isP\cdot\Theta}\int_{0}^{+\infty}u^{\frac{d-1}{\gamma_{P}}}%
G_{P}\left(  u^{1/\gamma_{P}}\right)  \eta\left(  u\right)  e^{-2\pi isu}du\\
&  +e^{-2\pi isQ\cdot\Theta}\int_{0}^{+\infty}u^{\frac{d-1}{\gamma_{Q}}}%
G_{Q}\left(  u^{1/\gamma_{Q}}\right)  \eta\left(  u\right)  e^{2\pi
isu}du+O\left(  \left\vert s\right\vert ^{-N}\right)  .
\end{align*}
It is enough to consider%
\[
K\left(  s\right)  =\int_{0}^{+\infty}u^{\frac{d-1}{\gamma}}G\left(
u^{1/\gamma}\right)  \eta\left(  u\right)  e^{-2\pi isu}du.
\]
Since $G\left(  r\right)  $ is smooth, for every $N>0$ we can write the Taylor
expansion%
\begin{align*}
K\left(  s\right)   &  =\sum_{k=0}^{N-1}\frac{G^{\left(  k\right)  }\left(
0\right)  }{k!}\int_{0}^{+\infty}u^{\frac{d-1+k}{\gamma}}\eta\left(  u\right)
e^{-2\pi isu}du\\
&  +\int_{0}^{+\infty}u^{\frac{d-1+N}{\gamma}}G_{N}\left(  u^{1/\gamma
}\right)  \eta\left(  u\right)  e^{-2\pi isu}du.
\end{align*}
For $N$ large enough, the function $u^{\frac{d-1+N}{\gamma}}G_{N}\left(
u^{1/\gamma}\right)  \eta\left(  u\right)  $ has enough bounded derivatives so
that a repeated integration by parts gives%
\[
\left\vert \int_{0}^{+\infty}\eta\left(  u\right)  G_{N}\left(  u^{1/\gamma
}\right)  u^{\frac{d-1+N}{\gamma}}e^{-2\pi isu}du\right\vert \leqslant
c\left\vert s\right\vert ^{-1-\frac{d}{\gamma}}.
\]
Finally, all other terms in the above sum have the form%
\[
\int_{0}^{+\infty}\eta\left(  u\right)  u^{\alpha}e^{-2\pi isu}du
\]
and can be estimated by the following lemma.
\end{proof}

\begin{lemma}
\label{Lemma alpha}If $\eta$ is as above then, for every $\alpha>-1$ and
$s\neq0$, we have%
\[
\int_{0}^{+\infty}t^{\alpha}e^{-2\pi ist}\eta\left(  t\right)  dt=\frac
{\Gamma\left(  \alpha+1\right)  }{\left(  2\pi\left\vert s\right\vert \right)
^{\alpha+1}}e^{-i\frac{\pi}{2}\left(  \alpha+1\right)  \operatorname{sgn}%
\left(  s\right)  }+O\left(  \left\vert s\right\vert ^{-N}\right)  .
\]

\end{lemma}

The above result is not surprising since, in the sense of distributions,%
\[
\int_{0}^{+\infty}t^{\alpha}e^{-2\pi ist}dt=\frac{\Gamma\left(  \alpha
+1\right)  }{\left(  2\pi\left\vert s\right\vert \right)  ^{\alpha+1}%
}e^{-i\frac{\pi}{2}\left(  \alpha+1\right)  \operatorname{sgn}\left(
s\right)  }.
\]
See e.g. \cite{G-S}. The following is a direct proof.

\begin{proof}
Assume first $s>0$. An integration by parts gives%
\begin{align*}
&  \int_{0}^{+\infty}t^{\alpha}e^{-2\pi ist}\eta\left(  t\right)  dt=\frac
{1}{2\pi is}\int_{0}^{+\infty}e^{-2\pi ist}\frac{d}{dt}\left[  t^{\alpha}%
\eta\left(  t\right)  \right]  dt\\
=  &  \frac{\alpha}{2\pi is}\int_{0}^{+\infty}e^{-2\pi ist}t^{\alpha-1}%
\eta\left(  t\right)  dt+\frac{1}{2\pi is}\int_{0}^{+\infty}e^{-2\pi
ist}t^{\alpha}\eta^{\prime}\left(  t\right)  dt.
\end{align*}
Since $\operatorname*{supp}\eta^{\prime}\subset\left(  \varepsilon
,2\varepsilon\right)  $ the term $t^{\alpha}\eta^{\prime}\left(  t\right)  $
is smooth so that%
\[
\frac{1}{2\pi is}\int_{0}^{+\infty}e^{-2\pi ist}t^{\alpha}\eta^{\prime}\left(
t\right)  dt=O\left(  \left\vert s\right\vert ^{-N}\right)  .
\]
Repeating the integration by parts $k$ times, with $k\leqslant\alpha$ gives%
\begin{align*}
&  \int_{0}^{+\infty}t^{\alpha}e^{-2\pi ist}\eta\left(  t\right)  dt\\
&  =\frac{\alpha\left(  \alpha-1\right)  \cdots\left(  \alpha-k+1\right)
}{\left(  2\pi is\right)  ^{k}}\int_{0}^{+\infty}e^{-2\pi ist}t^{\alpha-k}%
\eta\left(  t\right)  dt+O\left(  \left\vert s\right\vert ^{-N}\right)  .
\end{align*}
Assume first that $\alpha$ is an integer and take $k=\alpha$. Then%
\begin{align*}
\int_{0}^{+\infty}t^{\alpha}e^{-2\pi ist}\eta\left(  t\right)  dt  &
=\frac{\alpha!}{\left(  2\pi is\right)  ^{\alpha}}\int_{0}^{+\infty}e^{-2\pi
ist}\eta\left(  t\right)  dt+O\left(  \left\vert s\right\vert ^{-N}\right) \\
&  =\frac{\alpha!}{\left(  2\pi is\right)  ^{\alpha+1}}+\frac{\alpha!}{\left(
2\pi is\right)  ^{\alpha+1}}\int_{0}^{+\infty}e^{-2\pi ist}\eta^{\prime
}\left(  t\right)  dt\\
&  =\frac{\alpha!}{\left(  2\pi is\right)  ^{\alpha+1}}+O\left(  \left\vert
s\right\vert ^{-N}\right)  .
\end{align*}
If $\alpha$ is not an integer we take $k=\left[  \alpha\right]  +1$. Then%
\begin{align*}
&  \int_{0}^{+\infty}t^{\alpha}e^{-2\pi ist}\eta\left(  t\right)  dt\\
&  =\frac{\alpha\left(  \alpha-1\right)  \cdots\left(  \alpha-\left[
\alpha\right]  \right)  }{\left(  2\pi is\right)  ^{\left[  \alpha\right]
+1}}\int_{0}^{+\infty}e^{-2\pi ist}t^{\lambda-1}\eta\left(  t\right)
dt+O\left(  \left\vert s\right\vert ^{-N}\right)  ,
\end{align*}
where $\lambda=\alpha-\left[  \alpha\right]  $. By \cite[(4) pag. 48]{E} we
have%
\begin{align*}
&  \int_{0}^{+\infty}e^{-2\pi ist}t^{\lambda-1}\eta\left(  t\right)
dt=\overline{\int_{0}^{+\infty}e^{2\pi ist}t^{\lambda-1}\eta\left(  t\right)
dt}\\
=  &  -\sum_{n=0}^{N-1}\frac{\Gamma\left(  n+\lambda\right)  }{n!}e^{-\pi
i\left(  n+\lambda-2\right)  /2}\eta^{\left(  n\right)  }\left(  0\right)
\left(  2\pi s\right)  ^{-n-\lambda}+O\left(  s^{-N}\right)
\end{align*}
and since $\eta^{\left(  n\right)  }\left(  0\right)  =0$ for every $n>0$ and
$\eta\left(  0\right)  =1$ we obtain%
\begin{align*}
&  \int_{0}^{+\infty}t^{\alpha}e^{-2\pi ist}\eta\left(  t\right)  dt\\
=  &  \frac{\alpha\left(  \alpha-1\right)  \cdots\left(  \alpha-\left[
\alpha\right]  \right)  }{\left(  2\pi is\right)  ^{\left[  \alpha\right]
+1}}\int_{0}^{+\infty}e^{-2\pi ist}t^{\lambda-1}\eta\left(  t\right)
dt+O\left(  \left\vert s\right\vert ^{-N}\right) \\
=  &  \frac{\Gamma\left(  \alpha+1\right)  }{\left(  2\pi s\right)
^{\alpha+1}}e^{-\frac{\pi}{2}i\left(  \alpha+1\right)  }+O\left(
s^{-N}\right)
\end{align*}
also in this case.

Let now $s<0$. Then%
\begin{align*}
\int_{0}^{+\infty}t^{\alpha}e^{-2\pi ist}\eta\left(  t\right)  dt  &
=\overline{\int_{0}^{+\infty}t^{\alpha}e^{-2\pi i\left(  -s\right)  t}%
\eta\left(  t\right)  dt}\\
&  =\frac{\Gamma\left(  \alpha+1\right)  }{\left(  2\pi\left\vert s\right\vert
\right)  ^{\alpha+1}}e^{i\frac{\pi}{2}\left(  \alpha+1\right)  }+O\left(
\left\vert s\right\vert ^{-N}\right)
\end{align*}

\end{proof}

In the next proposition we show that assumptions of Proposition
\ref{Prop Sezioni} are satisfied when the flat points are as in Proposition
\ref{Prop H}.

\begin{proposition}
\label{Prop asymp G}Let $\gamma>1$ and let $B$ be a bounded convex body in
$\mathbb{R}^{d}$. Let $U$ be a bounded open neighborhood of the origin in
$\mathbb{R}^{d-1}$ and let $H\left(  x\right)  \in C^{\infty}\left(  U\right)
$ such that $H\left(  0\right)  =0$, $\nabla H\left(  0\right)  =0$ and
$\operatorname{Hess}H\left(  0\right)  $ positive definite (see Proposition
\ref{Prop H}). Assume there exists a neighborhood of the origin $W\subset
\mathbb{R}^{d}$ such that, in suitable coordinates,%
\[
\partial B\cap W=\left\{  \left(  x,t\right)  \in\mathbb{R}^{d}:t=\left(
H\left(  x\right)  \right)  ^{\gamma/2}\right\}  \cap W.
\]
As before, let%
\[
S\left(  t\right)  =\left\vert \left\{  x\in\mathbb{R}^{d-1}:\left(
x,t\right)  \in B\right\}  \right\vert .
\]
Then, there exists a smooth function $G\left(  r\right)  $ such that for $t>0$
sufficiently small we have%
\[
S\left(  t\right)  =t^{\frac{d-1}{\gamma}}G\left(  t^{1/\gamma}\right)
\]
with $G\left(  0\right)  $ equal to the $\left(  d-1\right)  $-dimensional
measure of the ellipsoid%
\[
\left\{  x\in\mathbb{R}^{d-1}:\frac{1}{2}\sum_{j,k=1}^{d-1}\frac{\partial
^{2}H}{\partial x_{j}\partial x_{k}}\left(  0\right)  x_{j}x_{k}%
\leqslant1\right\}  .
\]

\end{proposition}

\begin{proof}
For $t$ small enough we have%
\[
S\left(  t\right)  =\int_{\left\{  x\in\mathbb{R}^{d-1}:t\geqslant\left(
H\left(  x\right)  \right)  ^{\gamma/2}\right\}  }dx.
\]
By Morse's lemma (see \cite[p. 346]{Stein}), there exists a diffeomorphism
$\Psi\left(  y\right)  $ between two small neighborhoods of the origin in
$\mathbb{R}^{d-1}$ such that%
\[
H\left(  \Psi\left(  y\right)  \right)  =\left\vert y\right\vert ^{2}.
\]
Then,%
\begin{align*}
S\left(  t\right)   &  =\int_{\left\{  x\in\mathbb{R}^{d-1}:t\geqslant\left(
H\left(  x\right)  \right)  ^{\gamma/2}\right\}  }dx=\int_{\left\{  \left\vert
y\right\vert \leqslant t^{1/\gamma}\right\}  }J_{\Psi}\left(  y\right)  dy\\
&  =t^{\frac{d-1}{\gamma}}\int_{\left\{  \left\vert u\right\vert
\leqslant1\right\}  }J_{\Phi}\left(  t^{1/\gamma}u\right)  du=t^{\frac
{d-1}{\gamma}}G\left(  t^{1/\gamma}\right)
\end{align*}
where%
\[
G\left(  r\right)  =\int_{\left\{  \left\vert z\right\vert \leqslant1\right\}
}J_{\Phi}\left(  ru\right)  du.
\]
Finally observe that%
\begin{align*}
G\left(  0\right)   &  =\lim_{r\rightarrow0}\int_{\left\{  \left\vert
u\right\vert \leqslant1\right\}  }J_{\Phi}\left(  ru\right)  du=\lim
_{r\rightarrow0}\frac{1}{r^{d-1}}\int_{\left\{  \left\vert w\right\vert
\leqslant r\right\}  }J_{\Phi}\left(  w\right)  dw\\
&  =\lim_{r\rightarrow0}\frac{1}{r^{d-1}}\int_{\left\{  \left\vert
w\right\vert ^{2}\leqslant r^{2}\right\}  }J_{\Phi}\left(  w\right)  dw\\
&  =\lim_{r\rightarrow0}\frac{1}{r^{d-1}}\int_{\left\{  H\left(  y\right)
\leqslant r^{2}\right\}  }dy=\lim_{r\rightarrow0}\int_{\left\{  \frac{H\left(
rx\right)  }{r^{2}}\leqslant1\right\}  }dx\\
&  =\left\vert \left\{  x\in\mathbb{R}^{d-1}:\frac{1}{2}\sum_{j,k=1}%
^{d-1}\frac{\partial^{2}H}{\partial x_{j}\partial x_{k}}\left(  0\right)
x_{j}x_{k}\leqslant1\right\}  \right\vert .
\end{align*}

\end{proof}

\section{Proofs of the results}

\begin{proof}
[Proof of Proposition \ref{Prop H}]Let $\alpha$ be a multi-index. It is not
difficult to prove by induction on $\left\vert \alpha\right\vert $ that
\[
\frac{\partial^{\left\vert \alpha\right\vert }\Phi}{\partial x^{\alpha}%
}\left(  x\right)
\]
is a finite sum of terms of the form%
\[
c\left[  H\left(  x\right)  \right]  ^{\gamma/2-k}\frac{\partial^{\left\vert
\beta_{1}\right\vert }H}{\partial x_{\beta_{1}}}\left(  x\right)  \times
\cdots\times\frac{\partial^{\left\vert \beta_{k}\right\vert }H}{\partial
x_{\beta_{k}}}\left(  x\right)
\]
with $k\leqslant\left\vert \alpha\right\vert $ and multi-indices $\beta
_{1},\ldots,\beta_{k}$ such that $\left\vert \beta_{1}\right\vert
+\cdots+\left\vert \beta_{k}\right\vert =\left\vert \alpha\right\vert $. Since
$\operatorname{Hess}H\left(  0\right)  $ is positive definite there are
positive constants $c_{1}$ and $c_{2}$ such that in a neighborhood of the
origin%
\[
c_{1}\left\vert x\right\vert ^{2}\leqslant H\left(  x\right)  \leqslant
c_{2}\left\vert x\right\vert ^{2}%
\]
and
\[
\left\vert \frac{\partial H}{\partial x_{j}}\left(  x\right)  \right\vert
\leqslant c_{2}\left\vert x\right\vert .
\]
Moreover, since $H\left(  x\right)  $ is smooth%
\[
\left\vert \frac{\partial^{\left\vert \beta_{j}\right\vert }H}{\partial
x_{\beta_{j}}}\left(  x\right)  \right\vert \leqslant c\left\vert x\right\vert
^{\max\left(  2-\left\vert \beta_{j}\right\vert ,0\right)  }\leqslant
c\left\vert x\right\vert ^{2-\left\vert \beta_{j}\right\vert }.
\]
It follows that%
\begin{align*}
&  \left\vert \left[  H\left(  x\right)  \right]  ^{\gamma/2-k}\frac
{\partial^{\left\vert \beta_{1}\right\vert }H}{\partial x_{\beta_{1}}}\left(
x\right)  \times\cdots\times\frac{\partial^{\left\vert \beta_{k}\right\vert
}H}{\partial x_{\beta_{k}}}\left(  x\right)  \right\vert \\
&  \leqslant c\left(  \left\vert x\right\vert ^{2}\right)  ^{\gamma
/2-k}\left\vert x\right\vert ^{2-\left\vert \beta_{1}\right\vert }\times
\cdots\times\left\vert x\right\vert ^{2-\left\vert \beta_{k}\right\vert
}\leqslant c\left\vert x\right\vert ^{\gamma-2k}\left\vert x\right\vert
^{2k-\left(  \left\vert \beta_{1}\right\vert +\cdots+\left\vert \beta
_{k}\right\vert \right)  }\\
&  \leqslant c\left\vert x\right\vert ^{\gamma-\left\vert \alpha\right\vert }.
\end{align*}
This proves (\ref{Deriv-Phi}). To prove (\ref{Eigenvalues}) let us write%
\begin{align*}
&  \frac{\partial^{2}\Phi}{\partial x_{j}\partial x_{k}}\left(  x\right) \\
&  =\frac{\gamma}{2}\left(  \gamma/2-1\right)  \left[  H\left(  x\right)
\right]  ^{\gamma/2-2}\frac{\partial H}{\partial x_{j}}\left(  x\right)
\frac{\partial H}{\partial x_{k}}\left(  x\right)  +\frac{\gamma}{2}\left[
H\left(  x\right)  \right]  ^{\gamma/2-1}\frac{\partial^{2}H}{\partial
x_{j}\partial x_{k}}\left(  x\right) \\
&  =\frac{\gamma}{2}\left[  H\left(  x\right)  \right]  ^{\gamma/2-1}\left(
\frac{\partial^{2}H}{\partial x_{j}\partial x_{k}}\left(  x\right)  +\left(
\gamma/2-1\right)  \frac{\frac{\partial H}{\partial x_{j}}\left(  x\right)
\frac{\partial H}{\partial x_{k}}\left(  x\right)  }{H\left(  x\right)
}\right)
\end{align*}
so that%
\[
\operatorname{Hess}\Phi\left(  x\right)  =\frac{\gamma}{2}\left[  H\left(
x\right)  \right]  ^{\gamma/2-1}\operatorname{Hess}M\left(  x\right)
\]
where $M\left(  x\right)  $ is the matrix with entries%
\[
\frac{\partial^{2}H}{\partial x_{j}\partial x_{k}}\left(  x\right)  +\left(
\gamma/2-1\right)  \frac{\frac{\partial H}{\partial x_{j}}\left(  x\right)
\frac{\partial H}{\partial x_{k}}\left(  x\right)  }{H\left(  x\right)  }.
\]
Let $A=\operatorname{Hess}H\left(  0\right)  $, since%
\begin{align*}
H\left(  x\right)   &  =\frac{1}{2}x^{T}Ax+O\left(  \left\vert x\right\vert
^{3}\right)  ,\\
\nabla H\left(  x\right)   &  =Ax+O\left(  \left\vert x\right\vert ^{2}\right)
\\
\operatorname{Hess}H\left(  x\right)   &  =A+O\left(  \left\vert x\right\vert
\right)
\end{align*}
we have%
\begin{align*}
M\left(  x\right)   &  =A+O\left(  \left\vert x\right\vert \right)  +\left(
\gamma/2-1\right)  \frac{\left(  Ax+O\left(  \left\vert x\right\vert
^{2}\right)  \right)  \left(  Ax+O\left(  \left\vert x\right\vert ^{2}\right)
\right)  ^{T}}{\frac{1}{2}x^{T}Ax+O\left(  \left\vert x\right\vert
^{3}\right)  }\\
&  =A+\left(  \gamma-2\right)  \frac{Ax\left(  Ax\right)  ^{T}}{x^{T}%
Ax}+O\left(  \left\vert x\right\vert \right)  .
\end{align*}
Let us show that the matrix%
\[
A+\left(  \gamma-2\right)  \frac{Ax\left(  Ax\right)  ^{T}}{x^{T}Ax}%
\]
is positive definite. Indeed, for all $y\in\mathbb{R}^{d-1}$ we have%
\[
y^{T}\left(  A+\left(  \gamma-2\right)  \frac{Ax\left(  Ax\right)  ^{T}}%
{x^{T}Ax}\right)  y=y^{T}Ay+\left(  \gamma-2\right)  \frac{\left(
y^{T}Ax\right)  ^{2}}{x^{T}Ax}.
\]
When $\gamma\geqslant2$ we easily obtain%
\[
y^{T}\left(  A+\left(  \gamma-2\right)  \frac{Ax\left(  Ax\right)  ^{T}}%
{x^{T}Ax}\right)  y\geqslant y^{T}Ay\geqslant\lambda_{1}\left\vert
y\right\vert ^{2}%
\]
where $\lambda_{1}$ is the smallest eigenvalue of $A$. For $1<\gamma<2$, by
Cauchy-Schwarz inequality for the inner product defined by $\left\langle
y,x\right\rangle =y^{T}Ax$ we have
\[
\frac{\left(  y^{T}Ax\right)  ^{2}}{x^{T}Ax}\leqslant\frac{\left(
y^{T}Ay\right)  \left(  x^{T}Ax\right)  }{x^{T}Ax}=y^{T}Ay.
\]
Hence%
\begin{align*}
y^{T}\left(  A+\left(  \gamma-2\right)  \frac{Ax\left(  Ax\right)  ^{T}}%
{x^{T}Ax}\right)  y  &  \geqslant y^{T}Ay+\left(  \gamma-2\right)  y^{T}Ay\\
&  =\left(  \gamma-1\right)  y^{T}Ay\geqslant\left(  \gamma-1\right)
\lambda_{1}\left\vert y\right\vert ^{2}.
\end{align*}
Let $\mu_{1}\left(  x\right)  $ be the smallest eigenvalue of
$\operatorname{Hess}\Phi\left(  x\right)  .$ We want to show that $\mu
_{1}\left(  x\right)  \geqslant c\left\vert x\right\vert ^{\gamma-2}$. This is
equivalent to show that%
\[
y^{T}\operatorname{Hess}\Phi\left(  x\right)  y\geqslant c\left\vert
x\right\vert ^{\gamma-2}\left\vert y\right\vert ^{2}.
\]
We have%
\begin{align*}
y^{T}\operatorname{Hess}\Phi\left(  x\right)  y  &  =\frac{\gamma}{2}\left[
H\left(  x\right)  \right]  ^{\gamma/2-1}y^{T}\left(  A+\left(  \gamma
-2\right)  \frac{Ax\left(  Ax\right)  ^{T}}{x^{T}Ax}\right)  y\\
&  +\left[  H\left(  x\right)  \right]  ^{\gamma/2-1}y^{T}O\left(  \left\vert
x\right\vert \right)  y\\
&  \geqslant c_{1}\left(  \left\vert x\right\vert ^{2}\right)  ^{\gamma
/2-1}\left\vert y\right\vert ^{2}-c_{2}\left(  \left\vert x\right\vert
^{2}\right)  ^{\gamma/2-1}\left\vert x\right\vert \left\vert y\right\vert
^{2}\\
&  =c_{1}\left\vert x\right\vert ^{\gamma-2}\left\vert y\right\vert ^{2}%
-c_{2}\left\vert x\right\vert ^{\gamma-2}\left\vert x\right\vert \left\vert
y\right\vert ^{2}\geqslant c\left\vert x\right\vert ^{\gamma-2}\left\vert
y\right\vert ^{2}%
\end{align*}
for $\left\vert x\right\vert $ small enough.
\end{proof}

\bigskip

To prove the theorems and the corollary it is convenient to introduce a
mollified discrepancy.

\begin{lemma}
\label{Lemma Mollif}Let $\varphi\left(  \mathbf{z}\right)  $ be a compactly
supported smooth function in $\mathbb{R}^{d}$ with integral $1$. Then, if the
support of $\varphi\left(  \mathbf{z}\right)  $ is sufficiently small, for
every $0<\varepsilon<1$ and $R>1$ we have%
\[
\varepsilon^{-d}\varphi\left(  \varepsilon^{-1}\cdot\right)  \ast\chi_{\left(
R-\varepsilon\right)  B}\left(  \mathbf{z}\right)  \leqslant\chi_{RB}\left(
\mathbf{z}\right)  \leqslant\varepsilon^{-d}\varphi\left(  \varepsilon
^{-1}\cdot\right)  \ast\chi_{\left(  R+\varepsilon\right)  B}\left(
\mathbf{z}\right)  .
\]
In particular,%
\[
\left\vert B\right\vert \left(  \left(  R-\varepsilon\right)  ^{d}%
-R^{d}\right)  +D_{\varepsilon,R-\varepsilon}\left(  \mathbf{z}\right)
\leqslant D_{R}\left(  \mathbf{z}\right)  \leqslant\left\vert B\right\vert
\left(  \left(  R+\varepsilon\right)  ^{d}-R^{d}\right)  +D_{\varepsilon
,R+\varepsilon}\left(  \mathbf{z}\right)  ,
\]
where%
\[
D_{\varepsilon,R}\left(  \mathbf{z}\right)  =R^{d}\sum_{\mathbf{0}%
\neq\mathbf{m\in}\mathbb{Z}^{d}}\widehat{\varphi}\left(  \varepsilon
\mathbf{m}\right)  \widehat{\chi}_{B}\left(  R\mathbf{m}\right)  e^{2\pi
i\mathbf{m\cdot z}}.
\]

\end{lemma}

The above lemma is well known. See e.g. \cite[pag. 195]{BGiT} for a proof.

Also the following result is well known, the following is elementary proof.

\begin{lemma}
\label{Vandermonde}For every integer $M>0$ and every neighborhood $U$ of the
origin in $\mathbb{R}^{d}$ there exists a smooth function $\varphi\left(
\mathbf{z}\right)  $ supported in $U$ such that%
\[
\widehat{\varphi}\left(  \mathbf{0}\right)  =1
\]
and for every multi-index $\alpha$, with $0<\left\vert \alpha\right\vert
\leqslant M$%
\[
\frac{\partial^{\left\vert \alpha\right\vert }\widehat{\varphi}}%
{\partial\mathbf{\zeta}^{\alpha}}\left(  \mathbf{ 0}\right)  =0.
\]

\end{lemma}

\begin{proof}
Let $\psi\left(  \mathbf{z}\right)  $ be a smooth function supported in $U$
such that%
\[
\int_{\mathbb{R}^{d}}\psi\left(  \mathbf{z}\right)  d\mathbf{z}=1.
\]
We want to find constants $c_{0},c_{1},\ldots,c_{M}$ such that the function
\[
\varphi\left(  \mathbf{z}\right)  =\sum_{k=0}^{M}2^{kd}c_{k}\psi\left(
2^{k}\mathbf{z}\right)
\]
satisfies the lemma. We have%
\[
\widehat{\varphi}\left(  \mathbf{\zeta}\right)  =\sum_{k=0}^{M}c_{k}%
\widehat{\psi}\left(  2^{-k}\mathbf{\zeta}\right)
\]
so that%
\[
\widehat{\varphi}\left(  0\right)  =\sum_{k=0}^{M}c_{k}\widehat{\psi}\left(
0\right)  =\sum_{k=0}^{M}c_{k}%
\]
and for every multi-index $\alpha$%
\[
\frac{\partial^{\left\vert \alpha\right\vert }\widehat{\varphi}}%
{\partial\mathbf{\zeta}^{\alpha}}\left(  0\right)  =\sum_{k=0}^{M}%
c_{k}2^{-k\left\vert \alpha\right\vert }\frac{\partial^{\left\vert
\alpha\right\vert }\widehat{\psi}}{\partial\mathbf{\zeta}^{\alpha}}\left(
0\right)  .
\]
Hence the coefficients $c_{k}$ are the solution of the non singular linear
system%
\[
\left\{
\begin{array}
[c]{l}%
c_{0}+c_{2}+\cdots+c_{M}=1\\
\left(  2^{-1}\right)  ^{0}c_{0}+\left(  2^{-1}\right)  ^{1}c_{1}%
+\cdots+\left(  2^{-1}\right)  ^{M}c_{M}=0\\
~~~~~~~\vdots\\
\left(  2^{-M}\right)  ^{0}c_{0}+\left(  2^{-M}\right)  ^{1}c_{1}%
+\cdots+\left(  2^{-M}\right)  ^{M}c_{M}=0
\end{array}
\right.
\]

\end{proof}

The following lemma collects the main estimates that we will use later.

\begin{lemma}
\label{Lemma Stime Discrep}Assume the inequalities%
\[
\left\vert \widehat{\chi}_{B}\left(  \mathbf{\zeta}\right)  \right\vert
\leqslant\left\{
\begin{array}
[c]{l}%
c\left\vert s\right\vert ^{-1-\frac{d-1}{\gamma}},\\
c\left\vert \mathbf{\xi}\right\vert ^{-\left(  d-1\right)  \frac{\gamma
-2}{2\left(  \gamma-1\right)  }}\left\vert s\right\vert ^{-\frac{d-1}{2\left(
\gamma-1\right)  }-1},\\
c\left\vert \mathbf{\xi}\right\vert ^{-\frac{d+1}{2}}.
\end{array}
\right.
\]
proved in Proposition \ref{Prop Stima chi}, where $\mathbf{\zeta}=\mathbf{\xi
}+s\Theta$, with $s=\mathbf{\zeta}\cdot\Theta$ and $\mathbf{\xi}\cdot
\Theta\mathbf{=}0$ for some $\Theta\in\mathbb{R}^{d}$ with $\left\vert
\Theta\right\vert =1$ and $\gamma>2$ and let $\varphi\left(  \mathbf{z}%
\right)  $ as in the previous lemma.

1) For every $\tau>0$ and $p>2d/\left(  d-1\right)  $ there exists $c$ such
that for every $\varepsilon>0$ and $R>1$,%
\[
\left(  \int_{\mathbb{T}^{d}}\left\vert R^{d}\sum_{\left\vert \mathbf{m-}%
\left(  \mathbf{m}\cdot\Theta\right)  \Theta\right\vert \geqslant\tau}%
\widehat{\varphi}\left(  \varepsilon\mathbf{m}\right)  \widehat{\chi}%
_{B}\left(  R\mathbf{m}\right)  e^{2\pi i\mathbf{m\cdot z}}\right\vert
^{p}d\mathbf{z}\right)  ^{1/p}\leqslant cR^{\frac{d-1}{2}}\varepsilon
^{-\frac{d-1}{2}+\frac{d}{p}}.
\]

2) For every $\tau>0$ there exists $c$ such that for every $\varepsilon>0$,
$R>1$ and $\mathbf{z}\in\mathbb{T}^{d}$,%
\[
\left\vert R^{d}\sum_{\mathbf{0}\neq\mathbf{m\in}\mathbb{Z}^{d},~~\left\vert
\mathbf{m-}\left(  \mathbf{m}\cdot\Theta\right)  \Theta\right\vert <\tau
}\widehat{\varphi}\left(  \varepsilon\mathbf{m}\right)  \widehat{\chi}%
_{B}\left(  R\mathbf{m}\right)  e^{2\pi i\mathbf{m\cdot z}}\right\vert
\leqslant cR^{\left(  d-1\right)  \left(  1-\frac{1}{\gamma}\right)  }.
\]

\end{lemma}

\begin{proof}
Let us prove 1). For every $\mathbf{m}\in\mathbb{Z}^{d}$, write $\mathbf{m}%
=\mathbf{m}_{1}+\mathbf{m}_{2}$ with $\mathbf{m}_{1}=\mathbf{m-}\left(
\mathbf{m}\cdot\Theta\right)  \Theta$ and $\mathbf{m}_{2}=\left(
\mathbf{m}\cdot\Theta\right)  \Theta$. Also observe that for every $M>0$,
\[
\left\vert \widehat{\varphi}\left(  \mathbf{\zeta}\right)  \right\vert
\leqslant c_{M}\left(  1+\left\vert \mathbf{\zeta}\right\vert \right)  ^{-M}.
\]
Since $p\geqslant2$, by the Hausdorff-Young inequality with $1/p+1/q=1$ and
the assumption on $\widehat{\chi}_{B}\left(  \mathbf{\zeta}\right)  $ we have%
\begin{align*}
&  \left(  \int_{\mathbb{T}^{d}}\left\vert R^{d}\sum_{\left\vert
\mathbf{m}_{1}\right\vert \geqslant\tau}\widehat{\varphi}\left(
\varepsilon\mathbf{m}\right)  \widehat{\chi}_{B}\left(  R\mathbf{m}\right)
e^{2\pi i\mathbf{m\cdot z}}\right\vert ^{p}d\mathbf{z}\right)  ^{q/p}\\
&  \leqslant R^{dq}\sum_{\left\vert \mathbf{m}_{1}\right\vert \geqslant\tau
}\left\vert \widehat{\varphi}\left(  \varepsilon\mathbf{m}\right)  \right\vert
^{q}\left\vert \widehat{\chi}_{B}\left(  R\mathbf{m}\right)  \right\vert
^{q}\\
&  \leqslant cR^{q\frac{d-1}{2}}\sum_{\tau\leqslant\left\vert \mathbf{m}%
_{1}\right\vert \leqslant\left\vert \mathbf{m}_{2}\right\vert }\left(
1+\varepsilon\left\vert \mathbf{m}_{2}\right\vert \right)  ^{-M}\left\vert
\mathbf{m}_{1}\right\vert ^{-q\left(  d-1\right)  \frac{\gamma-2}{2\left(
\gamma-1\right)  }}\left\vert \mathbf{m}_{2}\right\vert ^{-q\frac
{d-1}{2\left(  \gamma-1\right)  }-q}\\
&  +cR^{q\frac{d-1}{2}}\sum_{\left\vert \mathbf{m}_{1}\right\vert >\max\left(
\left\vert \mathbf{m}_{2}\right\vert ,\tau\right)  }\left(  1+\varepsilon
\left\vert \mathbf{m}_{1}\right\vert \right)  ^{-M}\left\vert \mathbf{m}%
_{1}\right\vert ^{-\frac{d+1}{2}q}\\
&  =A+B.
\end{align*}

Since in the series in $A$ the quantities $\left\vert \mathbf{m}%
_{1}\right\vert $ and $\left\vert \mathbf{m}_{2}\right\vert $ are bounded away
from zero we can control the series with an integral,%
\begin{align*}
&  \sum_{\tau<\left\vert \mathbf{m}_{1}\right\vert \leqslant\left\vert
\mathbf{m}_{2}\right\vert }\left(  1+\varepsilon\left\vert \mathbf{m}%
_{2}\right\vert \right)  ^{-M}\left\vert \mathbf{m}_{1}\right\vert ^{-q\left(
d-1\right)  \frac{\gamma-2}{2\left(  \gamma-1\right)  }}\left\vert
\mathbf{m}_{2}\right\vert ^{-q\frac{d-1}{2\left(  \gamma-1\right)  }-q}\\
&  \leqslant c\iint\nolimits_{\left\{  \tau<\left\vert \xi\right\vert
\leqslant\left\vert s\right\vert \right\}  }\left(  1+\varepsilon\left\vert
s\right\vert \right)  ^{-M}\left\vert \xi\right\vert ^{-q\left(  d-1\right)
\frac{\gamma-2}{2\left(  \gamma-1\right)  }}\left\vert s\right\vert
^{-q\frac{d-1}{2\left(  \gamma-1\right)  }-q}d\xi ds\\
&  \leqslant c\int_{\tau}^{+\infty}\left(  1+\varepsilon\left\vert
s\right\vert \right)  ^{-M}\left\vert s\right\vert ^{-q\frac{d-1}{2\left(
\gamma-1\right)  }-q}\left[  \int_{\left\{  \left\vert \xi\right\vert
\leqslant\left\vert s\right\vert \right\}  }\left\vert \xi\right\vert
^{-q\left(  d-1\right)  \frac{\gamma-2}{2\left(  \gamma-1\right)  }}%
d\xi\right]  ds\\
&  \leqslant c\int_{\tau}^{+\infty}\left(  1+\varepsilon\left\vert
s\right\vert \right)  ^{-M}\left\vert s\right\vert ^{d-1-q\frac{d+1}{2}%
}ds\leqslant c\varepsilon^{-\left(  d-q\frac{d+1}{2}\right)  }%
\end{align*}
(note that since $p>2d/\left(  d-1\right)  $ we have $q<2d/\left(  d+1\right)
$). Similarly, for the series in $B$,%
\begin{align*}
&  \sum_{\left\vert \mathbf{m}_{1}\right\vert >\max\left(  \left\vert
\mathbf{m}_{2}\right\vert ,\tau\right)  }\left(  1+\varepsilon\left\vert
\mathbf{m}_{1}\right\vert \right)  ^{-M}\left\vert \mathbf{m}_{1}\right\vert
^{-\frac{d+1}{2}q}\\
&  \leqslant c%
{\displaystyle\iint\nolimits_{\left\{  \left\vert \xi\right\vert >\left\vert
s\right\vert \right\}  }}
\left(  1+\varepsilon\left\vert \xi\right\vert \right)  ^{-M}\left\vert
\xi\right\vert ^{-\frac{d+1}{2}q}d\xi ds\\
&  =c\varepsilon^{\frac{d+1}{2}q-d}\int_{\mathbb{R}^{d-1}}\left(  1+\left\vert
\xi\right\vert \right)  ^{-M}\left\vert \xi\right\vert ^{1-\frac{d+1}{2}q}%
d\xi=c\varepsilon^{\frac{d+1}{2}q-d}.
\end{align*}
This proves point 1) in the statement. Similarly, to prove point 2) observe
that, by the assumption on $\widehat{\chi}_{B}\left(  \mathbf{\zeta}\right)
$, we have%
\begin{align*}
\left\vert R^{d}\sum_{\mathbf{0}\neq\mathbf{m\in}\mathbb{Z}^{d},\left\vert
\mathbf{m}_{1}\right\vert <\tau}\widehat{\varphi}\left(  \varepsilon
\mathbf{m}\right)  \widehat{\chi}_{B}\left(  R\mathbf{m}\right)  e^{2\pi
i\mathbf{m\cdot z}}\right\vert  &  \leqslant R^{d}\sum_{\mathbf{0}%
\neq\mathbf{m\in}\mathbb{Z}^{d},\left\vert \mathbf{m}_{1}\right\vert <\tau
}\left\vert \widehat{\chi}_{B}\left(  R\mathbf{m}\right)  \right\vert \\
&  \leqslant cR^{\left(  d-1\right)  \left(  1-\frac{1}{\gamma}\right)  }%
\sum_{\mathbf{0}\neq\mathbf{m\in}\mathbb{Z}^{d},\left\vert \mathbf{m}%
_{1}\right\vert <\tau}\left\vert \mathbf{m}_{2}\right\vert ^{-1-\frac
{d-1}{\gamma}}.
\end{align*}
Note that the last series is essentially one dimensional and it is convergent.
\end{proof}

\bigskip

\begin{proof}
[Proof of Theorem \ref{Theorem A}]The discrepancy will be estimated using the
size of $\widehat{\chi}_{B}\left(  \mathbf{\zeta}\right)  $. Since the main
contribution to the size of this Fourier transform comes from the flat points
on $\partial B$ and since with a suitable partition of unity we can isolate
such flat points, without loss of generality we can assume the existence of
only a single flat point of order $\gamma$.

The case $1<\gamma\leqslant2$ follows from the argument used in \cite{BCGT}
for the smooth case. This essentially reduces to the Hausdorff-Young
inequality and follows from the estimate%
\[
\left\vert \widehat{\chi}_{B}\left(  \mathbf{\zeta}\right)  \right\vert
\leqslant c\left\vert \mathbf{\zeta}\right\vert ^{-\frac{d+1}{2}}%
\]
that holds true also in our case by (\ref{Stima Chi ottimale}). Let us now
prove point 2) and point 3) in the theorem. To prove point 2)\ we observe that
the case $p\leqslant2d/\left(  d+1-\gamma\right)  $ follows from the case
$p=2d/\left(  d+1-\gamma\right)  $, and the case $2d/\left(  d+1-\gamma
\right)  \leqslant p\leqslant+\infty$ follows by interpolation between
$p=2d/\left(  d+1-\gamma\right)  $ and $p=+\infty$. Hence to prove point 2) it
suffices to consider only the cases $p=2d/\left(  d+1-\gamma\right)  $ and
$p=+\infty$. Similarly to prove point 3) it suffices to consider only the case
$p=+\infty$. Observe that since $\gamma>2$, all these values of $p$ are
greater than $2d/\left(  d-1\right)  $.

By Lemma \ref{Lemma Mollif} we have%
\begin{align*}
\left\Vert D_{R}\right\Vert _{L^{p}\left(  \mathbb{T}^{d}\right)  }  &
\leqslant\left\vert B\right\vert \max_{\pm}\left\vert \left(  R\pm
\varepsilon\right)  ^{d}-R^{d}\right\vert +\max_{\pm}\left\Vert D_{\varepsilon
,R\pm\varepsilon}\right\Vert _{L^{p}\left(  \mathbb{T}^{d}\right)  }\\
&  \leqslant cR^{d-1}\varepsilon+\max_{\pm}\left\Vert D_{\varepsilon
,R\pm\varepsilon}\right\Vert _{L^{p}\left(  \mathbb{T}^{d}\right)  }.
\end{align*}

Replacing $R\pm\varepsilon$ with $R$ for simplicity, Lemma
\ref{Lemma Stime Discrep}, with a fixed $\tau>0$, gives%
\begin{align*}
\left\Vert D_{\varepsilon,R}\right\Vert _{L^{p}\left(  \mathbb{T}^{d}\right)
}  &  =\left(  \int_{\mathbb{T}^{d}}\left\vert R^{d}\sum_{\mathbf{0}%
\neq\mathbf{m}\in\mathbb{Z}^{d}}\widehat{\varphi}\left(  \varepsilon
\mathbf{m}\right)  \widehat{\chi}_{B}\left(  R\mathbf{m}\right)  e^{2\pi
i\mathbf{m\cdot z}}\right\vert ^{p}d\mathbf{z}\right)  ^{1/p}\\
&  \leqslant\left(  \int_{\mathbb{T}^{d}}\left\vert R^{d}\sum_{\left\vert
\mathbf{m-}\left(  \mathbf{m}\cdot\Theta\right)  \Theta\right\vert
\geqslant\tau}\widehat{\varphi}\left(  \varepsilon\mathbf{m}\right)
\widehat{\chi}_{B}\left(  R\mathbf{m}\right)  e^{2\pi i\mathbf{m\cdot z}%
}\right\vert ^{p}d\mathbf{z}\right)  ^{1/p}\\
&  +\left\vert R^{d}\sum_{\mathbf{0}\neq\mathbf{m\in}\mathbb{Z}^{d}%
,~~\left\vert \mathbf{m-}\left(  \mathbf{m}\cdot\Theta\right)  \Theta
\right\vert <\tau}\widehat{\varphi}\left(  \varepsilon\mathbf{m}\right)
\widehat{\chi}_{B}\left(  R\mathbf{m}\right)  e^{2\pi i\mathbf{m\cdot z}%
}\right\vert \\
&  \leqslant cR^{\frac{d-1}{2}}\varepsilon^{-\frac{d-1}{2}+\frac{d}{p}%
}+cR^{\left(  d-1\right)  \left(  1-\frac{1}{\gamma}\right)  }%
\end{align*}

The choice $\varepsilon=R^{-\frac{d-1}{d+1-2d/p}}$ then gives%
\[
\left\Vert D_{R}\right\Vert _{L^{p}\left(  \mathbb{T}^{d}\right)  }\leqslant
cR^{d\left(  d-1\right)  \frac{p-2}{p-2d+dp}}+cR^{\left(  d-1\right)  \left(
1-\frac{1}{\gamma}\right)  }.
\]
For $p=2d/\left(  d+1-\gamma\right)  $ and $2<\gamma\leqslant d+1$, or
$p=+\infty$ and $\gamma>d+1$ we obtain%
\[
\left\Vert D_{R}\right\Vert _{L^{p}\left(  \mathbb{T}^{d}\right)  }\leqslant
cR^{\left(  d-1\right)  \left(  1-\frac{1}{\gamma}\right)  }.
\]
For $2<\gamma<d+1$ and $p=+\infty$ we obtain%
\[
\left\Vert D_{R}\right\Vert _{L^{\infty}\left(  \mathbb{T}^{d}\right)
}\leqslant cR^{\frac{d\left(  d-1\right)  }{d+1}}.
\]

\end{proof}

\bigskip

\begin{proof}
[Proof of Theorem \ref{Theorem B}]The idea of the proof is simple. For every
$\mathbf{m}\in\mathbb{Z}^{d}$, write $\mathbf{m}=\mathbf{m}_{1}+\mathbf{m}%
_{2}$ with $\mathbf{m}_{1}=\mathbf{m-}\left(  \mathbf{m}\cdot\Theta\right)
\Theta$ and $\mathbf{m}_{2}=\left(  \mathbf{m}\cdot\Theta\right)  \Theta$, and
split the Fourier expansion of the discrepancy as%
\begin{equation}
D_{R}\left(  \mathbf{z}\right)  =R^{d}\sum_{\mathbf{m}_{1}=\mathbf{0,}\text{
}\mathbf{m}_{2}\neq\mathbf{0}}\widehat{\chi}_{B}\left(  R\mathbf{m}\right)
e^{2\pi i\mathbf{m}\cdot\mathbf{z}}+R^{d}\sum_{\mathbf{m}_{1}\neq\mathbf{0}%
}\widehat{\chi}_{B}\left(  R\mathbf{m}\right)  e^{2\pi i\mathbf{m}%
\cdot\mathbf{z}}. \label{spezzamento discrepanza}%
\end{equation}
We will see that the main term is the first one and it follows from
Proposition \ref{Prop Sezioni} that%
\begin{align*}
&  R^{d}\sum_{\mathbf{m}_{1}=\mathbf{0,}\text{ }\mathbf{m}_{2}\neq\mathbf{0}%
}\widehat{\chi}_{B}\left(  R\mathbf{m}\right)  e^{2\pi i\mathbf{m}%
\cdot\mathbf{z}}\\
&  \sim R^{\left(  d-1\right)  \left(  1-1/\gamma_{P}\right)  }A_{P}\left(
\mathbf{z-}RP\right)  +R^{\left(  d-1\right)  \left(  1-1/\gamma_{Q}\right)
}A_{Q}\left(  \mathbf{z-}RQ\right)  .
\end{align*}
The details are as follows. Let $D_{\varepsilon,R}\left(  \mathbf{z}\right)  $
be the mollified discrepancy as in the proof of Theorem \ref{Theorem A} and
let%
\[
Y\left(  \mathbf{z,}R\right)  =R^{\left(  d-1\right)  \left(  1-1/\gamma
_{P}\right)  }A_{P}\left(  \mathbf{z-}RP\right)  +R^{\left(  d-1\right)
\left(  1-1/\gamma_{Q}\right)  }A_{Q}\left(  \mathbf{z-}RQ\right)  .
\]
From Lemma \ref{Lemma Mollif} with a cut-off function as in Lemma
\ref{Vandermonde} we have%
\begin{align}
&  \left\vert D_{R}\left(  \mathbf{z}\right)  -Y\left(  \mathbf{z,}R\right)
\right\vert \nonumber\\
&  \leqslant\left\vert B\right\vert \max_{\pm}\left\vert \left(
R\pm\varepsilon\right)  ^{d}-R^{d}\right\vert +\max_{\pm}\left\vert
D_{\varepsilon,R\pm\varepsilon}\left(  \mathbf{z}\right)  -Y\left(
\mathbf{z,}R\pm\varepsilon\right)  \right\vert \label{NormaErrore}\\
&  +\left\vert Y\left(  \mathbf{z,}R\pm\varepsilon\right)  -Y\left(
\mathbf{z,}R\right)  \right\vert \nonumber
\end{align}
The first term in the right-hand side is bounded by $cR^{d-1}\varepsilon$. For
the third term we have%
\begin{align*}
&  \left\vert Y\left(  \mathbf{z,}R\pm\varepsilon\right)  -Y\left(
\mathbf{z,}R\right)  \right\vert \\
&  \leqslant\left\vert \left(  R\pm\varepsilon\right)  ^{\left(  d-1\right)
\left(  1-1/\gamma_{P}\right)  }A_{P}\left(  \mathbf{z}-\left(  R\pm
\varepsilon\right)  P\right)  -R^{\left(  d-1\right)  \left(  1-1/\gamma
_{P}\right)  }A_{P}\left(  \mathbf{z}-RP\right)  \right\vert \\
&  +\left\vert \left(  R\pm\varepsilon\right)  ^{\left(  d-1\right)  \left(
1-1/\gamma_{Q}\right)  }A_{Q}\left(  \mathbf{z}-\left(  R\pm\varepsilon
\right)  Q\right)  -R^{\left(  d-1\right)  \left(  1-1/\gamma_{Q}\right)
}A_{Q}\left(  \mathbf{z}-RQ\right)  \right\vert .
\end{align*}
The two terms are similar, let us consider only the first one. Then%
\begin{align*}
&  \left\vert \left(  R\pm\varepsilon\right)  ^{\left(  d-1\right)  \left(
1-1/\gamma_{P}\right)  }A_{P}\left(  \mathbf{z}-\left(  R\pm\varepsilon
\right)  P\right)  -R^{\left(  d-1\right)  \left(  1-1/\gamma_{P}\right)
}A_{P}\left(  \mathbf{z}-RP\right)  \right\vert \\
&  \leqslant\left(  R\pm\varepsilon\right)  ^{\left(  d-1\right)  \left(
1-1/\gamma_{P}\right)  }\left\vert A_{P}\left(  \mathbf{z}-\left(
R\pm\varepsilon\right)  P\right)  -A_{P}\left(  \mathbf{z}-RP\right)
\right\vert \\
&  +\left\vert \left(  R\pm\varepsilon\right)  ^{\left(  d-1\right)  \left(
1-1/\gamma_{P}\right)  }-R^{\left(  d-1\right)  \left(  1-1/\gamma_{P}\right)
}\right\vert \left\vert A_{P}\left(  \mathbf{z}-RP\right)  \right\vert
\end{align*}
Since%
\begin{align*}
&  \left\vert A_{P}\left(  \mathbf{z}-\left(  R\pm\varepsilon\right)
P\right)  -A_{P}\left(  \mathbf{z}-RP\right)  \right\vert \\
&  \leqslant c\sum_{k=1}^{+\infty}k^{-1-\frac{d-1}{\gamma_{P}}}\\
&  ~\times\left\vert \sin\left(  2\pi k\mathbf{m}_{0}\cdot\left(
\mathbf{z}-\left(  R\pm\varepsilon\right)  P\right)  -\frac{\pi}{2}\frac
{d-1}{\gamma_{P}}\right)  -\sin\left(  2\pi k\mathbf{m}_{0}\cdot\left(
\mathbf{z}-RP\right)  -\frac{\pi}{2}\frac{d-1}{\gamma_{P}}\right)  \right\vert
\\
&  \leqslant c\sum_{k=1}^{+\infty}k^{-1-\frac{d-1}{\gamma_{P}}}\left\vert
\sin\left(  \varepsilon\pi k\mathbf{m}_{0}\cdot P\right)  \right\vert \\
&  \leqslant c\varepsilon^{\frac{d-1}{\gamma_{P}}}.
\end{align*}
and $\left\vert A_{P}\left(  \mathbf{z}-RP\right)  \right\vert \leqslant c$ we
have%
\begin{align*}
&  \left\vert \left(  R\pm\varepsilon\right)  ^{\left(  d-1\right)  \left(
1-1/\gamma_{P}\right)  }A_{P}\left(  \mathbf{z}-\left(  R\pm\varepsilon
\right)  P\right)  -R^{\left(  d-1\right)  \left(  1-1/\gamma_{P}\right)
}A_{P}\left(  \mathbf{z}-RP\right)  \right\vert \\
&  \leqslant cR^{\left(  d-1\right)  \left(  1-1/\gamma_{P}\right)
}\varepsilon^{\frac{d-1}{\gamma_{P}}}+R^{\left(  d-1\right)  \left(
1-1/\gamma_{P}\right)  -1}\varepsilon.
\end{align*}
It remains to estimate the second term in (\ref{NormaErrore}) and for
simplicity we replace $R\pm\varepsilon$ with $R$. We have%
\begin{align}
&  D_{\varepsilon,R}\left(  \mathbf{z}\right)  -Y\left(  \mathbf{z},R\right)
\nonumber\\
&  =\left(  R^{d}\sum_{\mathbf{m}_{1}=\mathbf{0,m}_{2}\neq0}\widehat{\varphi
}\left(  \varepsilon\mathbf{m}\right)  \widehat{\chi}_{B}\left(
R\mathbf{m}\right)  e^{2\pi i\mathbf{m}\cdot\mathbf{z}}-Y\left(
\mathbf{z},R\right)  \right) \label{ErroreDaStimare}\\
&  +\left(  R^{d}\sum_{\mathbf{m}_{1}\neq\mathbf{0}}\widehat{\varphi}\left(
\varepsilon\mathbf{m}\right)  \widehat{\chi}_{B}\left(  R\mathbf{m}\right)
e^{2\pi i\mathbf{m}\cdot\mathbf{z}}\right) \nonumber\\
&  =I\left(  \mathbf{z}\right)  +II\left(  \mathbf{z}\right)  .\nonumber
\end{align}
For $I\left(  \mathbf{z}\right)  $ we have a pointwise estimate,%
\begin{align*}
\left\vert I\left(  \mathbf{z}\right)  \right\vert  &  \leqslant R^{d}%
\sum_{s\mathbf{\neq}0}\left\vert \widehat{\varphi}\left(  \varepsilon
s\mathbf{m}_{0}\right)  -1\right\vert \left\vert \widehat{\chi}_{B}\left(
Rs\mathbf{m}_{0}\right)  \right\vert \\
&  +\left\vert R^{d}\sum_{s\neq0}\widehat{\chi}_{B}\left(  Rs\mathbf{m}%
_{0}\right)  e^{2\pi is\mathbf{m}_{0}\cdot\mathbf{z}}-Y\left(  \mathbf{z}%
,R\right)  \right\vert .
\end{align*}
Our choice of the function $\varphi\left(  \mathbf{z}\right)  $ yields%
\[
\frac{d^{\left\vert \alpha\right\vert }\widehat{\varphi}}{d\zeta^{\alpha}%
}\left(  \mathbf{0}\right)  =0
\]
for every multi-index $\alpha$ with $\left\vert \alpha\right\vert \leqslant
M$. Hence
\[
\left\vert \widehat{\varphi}\left(  \mathbf{\zeta}\right)  -1\right\vert
\leqslant c_{M}\left\vert \mathbf{\zeta}\right\vert ^{M},
\]
and by Proposition \ref{Prop Stima chi} (recall that $\gamma_{P}%
\geqslant\gamma_{Q}$)%
\begin{align*}
&  R^{d}\sum_{\mathbf{s\neq}0}\left\vert \widehat{\varphi}\left(  \varepsilon
s\mathbf{m}_{0}\right)  -1\right\vert \left\vert \widehat{\chi}_{B}\left(
Rs\mathbf{m}_{0}\right)  \right\vert \\
&  \leqslant cR^{\left(  d-1\right)  \left(  1-\frac{1}{\gamma_{P}}\right)
}\sum_{s\neq0}\min\left(  \varepsilon^{M}\left\vert s\right\vert
^{M},1\right)  \left\vert s\right\vert ^{-1-\frac{d-1}{\gamma_{P}}}\\
&  \leqslant cR^{\left(  d-1\right)  \left(  1-\frac{1}{\gamma_{P}}\right)
}\sum_{\varepsilon\left\vert s\right\vert \leqslant1}\varepsilon^{M}\left\vert
s\right\vert ^{M}\left\vert s\right\vert ^{-1-\frac{d-1}{\gamma_{P}}%
}+cR^{\left(  d-1\right)  \left(  1-\frac{1}{\gamma_{P}}\right)  }%
\sum_{\varepsilon\left\vert s\right\vert >1}\left\vert s\right\vert
^{-1-\frac{d-1}{\gamma_{P}}}\\
&  \leqslant cR^{\left(  d-1\right)  \left(  1-\frac{1}{\gamma_{P}}\right)
}\varepsilon^{\frac{d-1}{\gamma_{P}}}.
\end{align*}
By our assumption on the direction $\Theta$ and by Proposition
\ref{Prop Sezioni} a long but direct computation gives%
\begin{align*}
&  R^{d}\sum_{\mathbf{m}_{1}=\mathbf{0,m}\neq\mathbf{0}}\widehat{\chi}%
_{B}\left(  R\mathbf{m}\right)  e^{2\pi i\mathbf{m}\cdot\mathbf{z}}=R^{d}%
\sum_{s\neq0}\widehat{\chi}_{B}\left(  Rs\mathbf{m}_{0}\right)  e^{2\pi
is\mathbf{m}_{0}\cdot\mathbf{z}}\\
&  =R^{\left(  d-1\right)  \left(  1-1/\gamma_{P}\right)  }A_{P}\left(
\mathbf{z-}RP\right)  +R^{\left(  d-1\right)  \left(  1-1/\gamma_{Q}\right)
}A_{Q}\left(  \mathbf{z-}RQ\right) \\
&  +O\left(  R^{d-1-\frac{d}{\gamma_{P}}}\right) \\
&  =Y\left(  \mathbf{z},R\right)  +O\left(  R^{d-1-\frac{d}{\gamma_{P}}%
}\right)
\end{align*}
Hence, we have the pointwise estimate%
\begin{equation}
\left\vert I\left(  \mathbf{z}\right)  \right\vert \leqslant cR^{\left(
d-1\right)  \left(  1-\frac{1}{\gamma_{P}}\right)  }\varepsilon^{\frac
{d-1}{\gamma_{P}}}+R^{d-1-\frac{d}{\gamma_{P}}}. \label{Norma_I}%
\end{equation}

The assumption that $\alpha\Theta\in\mathbb{Z}^{d}$ for some $\alpha$ implies
that the requirement $\mathbf{m}_{1}\neq\mathbf{0}$ is equivalent to
$\left\vert \mathbf{m}_{1}\right\vert \geqslant\tau$ for some $\tau>0$. By
Lemma \ref{Lemma Stime Discrep} we therefore have%
\begin{equation}
\left\Vert II\right\Vert _{p}\leqslant cR^{\frac{d-1}{2}}\varepsilon
^{-\frac{d-1}{2}+\frac{d}{p}}. \label{Norma_II}%
\end{equation}
Collecting the estimates (\ref{ErroreDaStimare}), (\ref{Norma_I}) and
(\ref{Norma_II}) we have%
\begin{align*}
&  \left\Vert D_{R}\left(  \mathbf{z}\right)  -Y\left(  \mathbf{z},R\right)
\right\Vert _{L^{p}\left(  \mathbb{T}^{d}\right)  }\\
&  \leqslant cR^{\left(  d-1\right)  \left(  1-1/\gamma_{P}\right)
}\varepsilon^{\frac{d-1}{\gamma_{P}}}+cR^{d-1}\varepsilon+cR^{d-1-\frac
{d}{\gamma_{P}}}+cR^{\frac{d-1}{2}}\varepsilon^{-\frac{d-1}{2}+d/p}.
\end{align*}
The choice $\varepsilon=R^{-\frac{d-1}{d+1-2d/p}}$ gives%
\begin{align*}
&  \left\Vert D_{R}\left(  \mathbf{z}\right)  -Y\left(  \mathbf{z},R\right)
\right\Vert _{L^{p}\left(  \mathbb{T}^{d}\right)  }\\
&  \leqslant cR^{\left(  d-1\right)  \left(  1-\frac{1}{\gamma_{P}}\right)
}\left(  R^{-\frac{d-1}{d+1-2d/p}\frac{d-1}{\gamma_{P}}}+R^{\frac{d-1}%
{\gamma_{P}}-\frac{d-1}{d+1-2d/p}}+R^{-\frac{1}{\gamma_{P}}}\right)  .
\end{align*}
Since our assumption implies $1/p>\left(  d+1-\gamma_{P}\right)  /\left(
2d\right)  $, all the exponents of $R$ in the parenthesis are negative and
therefore
\[
\left\Vert D_{R}\left(  \mathbf{z}\right)  -Y\left(  \mathbf{z},R\right)
\right\Vert _{L^{p}\left(  \mathbb{T}^{d}\right)  }\leqslant cR^{\left(
d-1\right)  \left(  1-\frac{1}{\gamma_{P}}\right)  -\delta}%
\]
for some $\delta>0$. This proves immediately point 2). It also prove point
1)\ as long as one notices that if $\gamma_{P}>\gamma_{Q}$%
\[
\left\Vert R^{\left(  d-1\right)  \left(  1-1/\gamma_{Q}\right)  }A_{Q}\left(
\mathbf{z-}RQ\right)  \right\Vert _{L^{p}\left(  \mathbb{T}^{d}\right)
}\leqslant cR^{\left(  d-1\right)  \left(  1-\frac{1}{\gamma_{P}}\right)
-\delta}%
\]
for a suitable $\delta>0$.
\end{proof}

\bigskip

\begin{proof}
[Proof of Corollary \ref{Corollary}]Because of our assumptions the constants
in front of the two series that define $A_{P}\left(  \mathbf{z-}RP\right)  $
and $A_{Q}\left(  \mathbf{z}-RQ\right)  $ are the same. A simple computation
gives%
\begin{align*}
&  A_{P}\left(  \mathbf{z-}RP\right)  +A_{Q}\left(  \mathbf{z}-RQ\right) \\
&  =\frac{4G_{P}\left(  0\right)  \Gamma\left(  \frac{d-1}{\gamma}+1\right)
}{\left(  2\pi\left\vert \mathbf{m}_{0}\right\vert \right)  ^{\frac
{d-1}{\gamma}+1}}\sum_{k=1}^{+\infty}k^{-1-\frac{d-1}{\gamma}}\sin\left(
\pi\left(  k\mathbf{m}_{0}\cdot R\left(  Q-P\right)  -\frac{d-1}{2\gamma
}\right)  \right) \\
&  \times\cos\left(  2\pi k\mathbf{m}_{0}\cdot\left(  \mathbf{z-}R\frac
{P+Q}{2}\right)  \right)  .
\end{align*}
and%
\[
A_{P}\left(  \mathbf{z-}RP\right)  +A_{Q}\left(  \mathbf{z}-RQ\right)  =0
\]
since $\mathbf{m}_{0}\cdot R\left(  Q-P\right)  $ and $\frac{d-1}{2\gamma}$
are integers.
\end{proof}

\bigskip

\begin{proof}
[Proof of Theorem \ref{Theorem C}]By Theorem 1.1 in \cite{BHI} we have the
following estimate for the $L^{2}$ average decay of the Fourier transform%
\[
\left(  \int_{SO\left(  d\right)  }\left\vert \widehat{\chi}_{B}\left(
R\sigma\mathbf{m}\right)  \right\vert ^{2}d\sigma\right)  ^{1/2}\leqslant
c\left(  R\left\vert \mathbf{m}\right\vert \right)  ^{-\frac{d+1}{2}}.
\]
Hence, applying the Hausdorff-Young inequality to (\ref{Fourier discrep}) with
$2\leqslant p<2d/\left(  d-1\right)  $ and $1/p+1/q=1$, we obtain%
\begin{align*}
\int_{SO\left(  d\right)  }\left(  \int_{\mathbb{T}^{d}}\left\vert
D_{R,\sigma}\left(  \mathbf{z}\right)  \right\vert ^{p}d\mathbf{z}\right)
^{q/p}d\sigma &  \leqslant R^{dq}\sum_{\mathbf{0}\neq\mathbf{m}\in
\mathbb{Z}^{d}}\int_{SO\left(  d\right)  }\left\vert \widehat{\chi}_{B}\left(
R\sigma\mathbf{m}\right)  \right\vert ^{q}d\sigma\\
&  \leqslant R^{dq}\sum_{\mathbf{0}\neq\mathbf{m}\in\mathbb{Z}^{d}}\left(
\int_{SO\left(  d\right)  }\left\vert \widehat{\chi}_{B}\left(  R\sigma
\mathbf{m}\right)  \right\vert ^{2}d\sigma\right)  ^{q/2}\\
&  \leqslant cR^{q\frac{d-1}{2}}\sum_{\mathbf{0}\neq\mathbf{m}\in
\mathbb{Z}^{d}}\left\vert \mathbf{m}\right\vert ^{-q\frac{d+1}{2}}\leqslant
cR^{q\frac{d-1}{2}}.
\end{align*}

In a similar way if $1\leqslant p\leqslant2$ we have,%
\begin{align*}
\int_{SO\left(  d\right)  }\left(  \int_{\mathbb{T}^{d}}\left\vert
D_{R,\sigma}\left(  \mathbf{z}\right)  \right\vert ^{p}d\mathbf{z}\right)
^{2/p}d\sigma &  \leqslant\int_{SO\left(  d\right)  }\int_{\mathbb{T}^{d}%
}\left\vert D_{R,\sigma}\left(  \mathbf{z}\right)  \right\vert ^{2}%
d\mathbf{z}d\sigma\\
&  =R^{2d}\sum_{\mathbf{0}\neq\mathbf{m}\in\mathbb{Z}^{d}}\int_{SO\left(
d\right)  }\left\vert \widehat{\chi}_{B}\left(  R\sigma\mathbf{m}\right)
\right\vert ^{2}d\sigma\\
&  \leqslant cR^{d-1}\sum_{\mathbf{0}\neq\mathbf{m}\in\mathbb{Z}^{d}%
}\left\vert \mathbf{m}\right\vert ^{-\left(  d+1\right)  }\leqslant cR^{d-1}.
\end{align*}

\end{proof}

\bigskip

\begin{proof}
[Proof of Theorem \ref{Theorem D}]Without loss of generality we can assume
$\left\vert \alpha\right\vert \leqslant\left\vert \beta\right\vert $. By
Proposition \ref{Prop Stima chi} we have%
\[
\left\vert \widehat{\chi}_{B}\left(  m,n\right)  \right\vert \leqslant\left\{
\begin{array}
[c]{l}%
c\left\vert \alpha m+\beta n\right\vert ^{-1-\frac{1}{\gamma}},\\
c\left\vert -\beta m+\alpha n\right\vert ^{-\frac{\gamma-2}{2\left(
\gamma-1\right)  }}\left\vert \alpha m+\beta n\right\vert ^{-\frac{1}{2\left(
\gamma-1\right)  }-1},\\
c\left\vert -\beta m+\alpha n\right\vert ^{-\frac{3}{2}}.
\end{array}
\right.
\]
We have%
\begin{align*}
\int_{\mathbb{T}^{d}}\left\vert D_{R}\left(  \mathbf{z}\right)  \right\vert
^{2}d\mathbf{z}  &  =R^{4}\sum_{\left(  m,n\right)  \neq\left(  0,0\right)
}\left\vert \widehat{\chi}_{B}\left(  Rm,Rn\right)  \right\vert ^{2}\\
&  \leqslant R^{4}\sum_{0<\left\vert -\beta m+\alpha n\right\vert
<1/2}\left\vert \widehat{\chi}_{B}\left(  Rm,Rn\right)  \right\vert ^{2}\\
&  ~~~~+R^{4}\sum_{1/2\leqslant\left\vert -\beta m+\alpha n\right\vert
\leqslant\left\vert \alpha m+\beta n\right\vert }\left\vert \widehat{\chi}%
_{B}\left(  Rm,Rn\right)  \right\vert ^{2}\\
&  ~~~~+R^{4}\sum_{0<\left\vert \alpha m+\beta n\right\vert <\left\vert -\beta
m+\alpha n\right\vert }\left\vert \widehat{\chi}_{B}\left(  Rm,Rn\right)
\right\vert ^{2}\\
&  =I+II+III.
\end{align*}
Using the above estimate for $\widehat{\chi}_{B}\left(  Rm,Rn\right)  $ we
have%
\[
III\leqslant cR\sum_{0<\left\vert \alpha m+\beta n\right\vert <\left\vert
-\beta m+\alpha n\right\vert }\left\vert \left(  m,n\right)  \right\vert
^{-3}\leqslant cR\sum_{\left(  m,n\right)  \neq\left(  0,0\right)  }\left\vert
\left(  m,n\right)  \right\vert ^{-3}\leqslant cR.
\]
In the term $II$ the quantity $\left\vert -\beta m+\alpha n\right\vert $ and
$\left\vert \alpha m+\beta n\right\vert $ are bounded away from zero so that,
arguing as in the proof of Lemma \ref{Lemma Stime Discrep}, we can replace the
series with the corresponding integral,%
\begin{align*}
II  &  \leqslant cR\sum_{1/2\leqslant\left\vert -\beta m+\alpha n\right\vert
<\left\vert \alpha m+\beta n\right\vert }\left\vert -\beta m+\alpha
n\right\vert ^{-\frac{\gamma-2}{\left(  \gamma-1\right)  }}\left\vert \alpha
m+\beta n\right\vert ^{-\frac{1}{\left(  \gamma-1\right)  }-2}\\
&  \leqslant cR\int_{\left\{  1/2\leqslant\left\vert \xi\right\vert
<\left\vert s\right\vert \right\}  }\left\vert \xi\right\vert ^{-\frac
{\gamma-2}{\left(  \gamma-1\right)  }}\left\vert s\right\vert ^{-\frac
{1}{\left(  \gamma-1\right)  }-2}dsd\xi\leqslant cR.
\end{align*}
In the term $I$ observe that $\left\vert -\beta m+\alpha n\right\vert <1/2$
implies $\left\vert \alpha m+\beta n\right\vert \approx\left\vert n\right\vert
$. Then%
\begin{align*}
I  &  \leqslant cR\sum_{\left\vert -\beta m+\alpha n\right\vert <1/2}%
\left\vert -\beta m+\alpha n\right\vert ^{-\frac{\gamma-2}{\gamma-1}%
}\left\vert \alpha m+\beta n\right\vert ^{-\frac{1}{\gamma-1}-2}\\
&  \leqslant cR\sum_{\left\vert -\beta m+\alpha n\right\vert <1/2}\left\Vert
\frac{\alpha}{\beta}n\right\Vert ^{-\frac{\gamma-2}{\gamma-1}}\left\vert
n\right\vert ^{-\frac{1}{\gamma-1}-2}\\
&  \leqslant cR\sum_{\left\vert -\beta m+\alpha n\right\vert <1/2}\left(
\left\vert n\right\vert ^{-1-\delta}\right)  ^{-\frac{\gamma-2}{\gamma-1}%
}\left\vert n\right\vert ^{-\frac{1}{\gamma-1}-2}\\
&  \leqslant cR\sum_{n=1}^{+\infty}n^{\left(  1+\delta\right)  \frac{\gamma
-2}{\gamma-1}}n^{-\frac{1}{\gamma-1}-2}\leqslant cR.
\end{align*}
In the last inequality we used the assumption $\delta<2/(\gamma-2).$
\end{proof}

\end{document}